\documentclass{amsart}
\usepackage{graphicx} % Required for inserting images
\usepackage{amsmath}
\usepackage{amssymb}
\usepackage{xcolor}
\usepackage{amsthm}
\usepackage{natbib}
\newtheorem{Prop}{Proposition}
\newtheorem{Thm}{Theorem}
\newtheorem{Def}{Definition}
\newtheorem{Lem}{Lemma}

\newtheorem{Rmk}{Remark}
\newtheorem*{proof-of-lemma4}{\rmfamily Proof of Lemma \ref{Lem:4}}
\newtheorem*{proof-of-lemma5}{\rmfamily Proof of Lemma \ref{Lem:5}}

\title[An inverse Coefficient problem for a semilinear wave equation]{An inverse coefficient problem for a semilinear wave equation by the first order linearization}
\author{Yuxiang He and Shuai Lu}
\address{School of Mathematical Sciences, Fudan University, 220 Handan Road, Shanghai 200433, China. Corresponding author: S. Lu}
\email{heyuxiang21@m.fudan.edu.cn (Y. He); slu@fudan.edu.cn (S. Lu) }
\date{\today}

\begin{document}

\maketitle

%\section{Update}

%Jan 25 2024: Reconstruction and stability in 3D situation added.

%July 18 2023: Using $\varepsilon$-expansion to derive the $u_1$ in chapter "Linearization", 

%July 12 2023: Correct the derivation of Born approximation, backpropagation algorithm and implementation, by introducing the initial incoming plane wave as in [Paul Muller, The Theory of Diffraction Tomography, 2016].

%July 9 2023: The first version.

\begin{abstract}
This paper investigates recovery of an unknown coefficient in a semilinear wave equation defined on a bounded, open, and strictly convex domain in \(\mathbb{R}^{1+n}\) with \(n \ge 2\). We demonstrate that the unknown coefficient \(q\) appearing in the semilinear wave equation \(\square u + q u^m = 0\) with Neumann boundary conditions can be reconstructed with H\"older stability from the linearized Neumann-to-Dirichlet (NtD) map. Our approach combines first-order linearization with the Principle of Inclusion-Exclusion (PIE) identity, and employs geometric optics solutions for wave equations in two distinct regimes: the case \(m=2\) with \(q = q(x)\), and the case \(m \ge 3\) with \(q = q(t,x)\). Furthermore, numerical examples illustrate that the unknown coefficient can also be effectively reconstructed using a neural network-based inversion algorithm within the framework of the least squares method.
\end{abstract}

\section{Introduction}
The inverse problem for wave equations represents a long-standing and central area of research within the inverse problems community. Early work in this field primarily addressed inverse problems for linear wave equations. Using the boundary control method, developed in \cite{Matti5,Matti6}, uniqueness results were established for recovering both a coefficient and the underlying Riemannian manifold with zero initial conditions. This method effectively combines wave propagation analysis with controllability theory %\cite{Katchalovetal.2001} 
and also offers an efficient framework for numerical implementation \cite{deHoopetal.2018}.
Several recent studies have further advanced the determination of unknown coefficients from boundary data for the linear wave equation and related equations, including contributions such as \cite{Anderson2004,Matti9,Kianetal.2019,Kurylevetal.2018b,Lassas2018}. %Helinetal.2018,Isozakietal.2017,Krupchyketal.2008,Matti26
Nevertheless, the classical boundary control method remains limited in scope: it applies either to time-independent coefficients or when the coefficients are real analytic in time \cite{Eskin2007}. In contrast, the approach presented in \cite{Alexakis2022,Alexakis2025} relies on a unique continuation property in the exterior region of double null cones. While conceptually related to the boundary control method, this technique generalizes the method to settings that do not require real analyticity in time.
We also note that various stability estimates for inverse coefficient problems for the linear wave equation are investigated in \cite{BaoZhang2014, Busch2026, Feizmohammadietal.2021, Stefanov1989, StefanovYang}.

%reconstructing a Riemannian metric \(g\) from the Dirichlet-to‑Neumann (DtN) map associated with the linear wave equation \((\partial_t^2-\Delta_g)u=0\) has been studied in several geometric settings.
%It is also worth noting that the reconstruction of potentials in linear versions of (\ref{2-1}), where $m=1$, has been studied in \cite{Feizmohammadietal.2021,Stefanov1989,StefanovYang%,Busch2026
%}. These works utilize the propagation of singularities along bicharacteristics to recover integrals of the coefficients over light rays, which requires knowledge of the DtN map or the scattering operator on the entire lateral boundary $(0,T)\times\partial\Omega$.

In recent years, the study of inverse problems for nonlinear wave equations has garnered growing interest. The pioneering work of \cite{Matti37} revealed that nonlinearity itself can serve as a powerful tool for solving corresponding inverse problems. By leveraging diverse nonlinear structures, several inverse problems that remain open in the linear setting have been successfully addressed in their nonlinear counterparts. The novel approach introduced in that work involves analyzing parameter-dependent families of solutions and performing higher-order linearizations with respect to each parameter.
This higher-order linearization technique, developed in \cite{Matti37}, has greatly enhanced the tractability of nonlinear inverse problems and has led to numerous advances in the field. For example, within the context of Einstein's equations, the recovery of leading-order coefficients has been investigated, beginning with the contributions in \cite{AliLauri14} and followed by further developments in \cite{AliLauri18, AliLauri24}.
More recently, as illustrated by works such as \cite{Matti3, Matti11, Matti12, AliLauri, KrupchykandUhlmann2020a, AliLauri14, LassasP8, Oksanenetal.2024, %SunandUhlmann1997, 
AliLauri24, WangandZhou2019}, inverse problems associated with nonlinear equations and higher-order linearization continue to be actively explored, with nonlinearity increasingly viewed as a beneficial resource.

Despite the existence of higher-order linearization methods, the conventional technique for solving inverse problems in nonlinear equations remains the first-order linearization. This approach was originally introduced in \cite{Isakov1993} for nonlinear parabolic inverse problems and has since been generalized to various inverse problems involving nonlinear models, as evidenced by works such as \cite{IsakovSylvester1994, LaiLin2019, SaloZhong2012, Sun2010, SunandUhlmann1997}.  
Recently, first-order linearization with recovery of the unknown  coefficient has been discussed in \cite{LuSaloXu, Zou} for solving the inverse Schr\"odinger potential problem for semilinear Helmholtz equation with power-type nonlinearity. Notably, the principle of inclusion-exclusion (PIE) is employed to prove that increasing stability holds for first-order linearization with respect to the potential function. Numerical examples in \cite{LuSaloXu, Zou} further demonstrate that first-order linearization is numerically more stable compared to higher-order linearization approaches, since numerical differentiation of first-order derivatives tends to be more stable in practice.

In this paper, we consider the following initial boundary value problem 
for the semilinear wave equation:
\begin{equation}\label{2-1}
    \begin{cases}
        \square u(t,x) 
        +qu^m(t,x) = 0, & \mathrm{in}\ (0,T)\times\Omega, \\
        \partial_\nu u(t,x) = f(t,x), & \mathrm{on}\ (0,T)\times\partial\Omega, \\
        u(0,x)=\partial_t u(0,x)=0, & x\in\Omega,
    \end{cases}
\end{equation}
where $m\ge2$ is an integer and $\Omega\subset\mathbb{R}^n$, with $n\ge2$,  
is a bounded open and strictly convex domain with smooth boundary 
$\partial\Omega$. Moreover, $\partial_\nu$ denotes the outer normal derivative on the 
boundary $(0,T)\times\partial\Omega$, and 
$\square =\partial_t^2 -\Delta $ is the standard 
wave operator.
For different nonlinear indices \( m \), we assume that the coefficient \( q \) satisfies the following property throughout the context:
\begin{equation*}
    q=\begin{cases}
        q(x)\in C^\infty(\overline{\Omega}),   & m=2, \\
        q(t,x)\in C^\infty([0,T]\times\overline{\Omega}), & m\ge3.   
    \end{cases}
\end{equation*}
We aim to establish a stability estimate for the recovery of the coefficient \( q \) in the above semilinear wave equation (\ref{2-1}) by combining the first-order linearization of the Neumann-to-Dirichlet (NtD) map and geometric optics constructions for wave equations.

\subsection{First-Order Linearization}
For ease of exposition, we first provide a description and formulation  of the first-order linearization. The rigorous proof will be presented in Section \ref{se2} below. Specifically, for any bounded Neumann boundary datum \( f \in C^{s+1}((0,T) \times \partial\Omega) \) with \( s \in \mathbb{N} \) satisfying \( s > (n+1)/2 \) and the compatibility conditions  
\[
\partial_t^l f\big|_{t=0} = 0, \quad \forall\, l = 0,1,\dots,s-1,
\]  
and for a sufficiently small coefficient \( q \), we obtain the well-posedness of problem (\ref{2-1}). Consequently, the following NtD map is well-defined:
\begin{Def}[Neumann-to-Dirichlet map]\label{Def:1}
    The \textbf{Neumann-to-Dirichlet map} $\Lambda_q$ of (\ref{2-1}) is defined by 
    \begin{equation}\label{eq_NtD map}
        \begin{split}
            \Lambda_q:\{ f\in 
    C^{s+1}((0,T)\times\partial\Omega): ||f||_{C^{s+1}}\le M\}
    &\to E^{s-\frac{1}{2}}((0,T)\times\partial\Omega) \\
  {\rm such\, that \qquad }  f &\mapsto u|_{(0,T)\times
        \partial\Omega}
        \end{split}
    \end{equation}
where $E^{s-\frac{1}{2}}((0,T)\times\partial\Omega)
:=\bigcap_{0\le k\le s} C^k((0,T);
    H^{s-k-\frac{1}{2}}(\partial\Omega))$ is the energy space and $H^s$ is the standard 
Sobolev space. 
\end{Def}

%\begin{equation}\label{2-2}\begin{split}
%    \Lambda_q:C^{s+1}((0,T)\times\partial\Omega)&\to 
%        E^{s-\frac{1}{2}}((0,T)\times\partial\Omega), \\
%    f &\mapsto u|_{(0,T)\times\partial\Omega},
%\end{split}\end{equation}

Following the first-order linearization for the semilinear parabolic and elliptic equations in \cite{Isakov1993,IsakovSylvester1994}, we outline the first-order linearization method for the semilinear wave equation (\ref{2-1}) below.  
Let $\gamma$ be a small parameter. Consider a scaled coefficient  
$\gamma q\in C^\infty(\overline{(0,T)\times\Omega})$ with  
$\|\gamma q\|_{C^\infty(\overline{(0,T)\times\Omega})}$ sufficiently small. The derivative of the NtD map given in (\ref{eq_NtD map}) with respect to \(\gamma\) yields the linearized NtD map:
\begin{equation}\label{11-3}
    \Lambda_q'f := \partial_\gamma(\Lambda_{\gamma q})|_{\gamma=0}f
    = u^{(1)}|_{(0,T)\times\partial\Omega},
\end{equation}
where $f\in C^{s+1}((0,T)\times\partial\Omega)$, $s>(n+1)/2$, and 
$u^{(1)}$ is the solution of the following linearized 
system:
\begin{equation}\label{11-5}
    (I_1)\begin{cases}
        \square u^{(1)} + q(u^{(0)})^m = 0, & \mathrm{in}\ (0,T)\times\Omega, \\
        \partial_\nu u^{(1)} = 0, & \mathrm{on}\ (0,T)\times\partial\Omega, \\
        u^{(1)}(0,x)=\partial_t u^{(1)}(0,x)=0, & x\in\Omega.
    \end{cases}
\end{equation}
Here, $u^{(0)}$ denotes the solution of 
the unperturbed problem:
\begin{equation}\label{11-4}
    (I_0)\begin{cases}
        \square u^{(0)}= 0, & \mathrm{in}\ (0,T)\times\Omega, \\
        \partial_\nu u^{(0)} = f, & \mathrm{on}\ (0,T)\times\partial\Omega, \\
        u^{(0)}(0,x)=\partial_t u^{(0)}(0,x)=0, & x\in\Omega.
    \end{cases}
\end{equation}

We study the inverse problem of 
reconstructing the unknown coefficient $q(x)$ or $q(t,x)$ from the knowledge of 
the first-order linearized NtD map $\Lambda_q'$, as defined in 
equation (\ref{11-3}). That is, we seek to identify the coefficient function $q$ by the linearized NtD map
\begin{equation*}
    q\mapsto\Lambda_q'.
\end{equation*}
We note that the linearized NtD map $\Lambda_q'$ is a nonlinear 
operator with respect to the Neumann boundary data $f$. By 
examining the linearized system (\ref{11-5})-(\ref{11-4}), one can 
verify that $\Lambda_q'$ is $m$-homogeneous, i.e.
\begin{equation*}
    \Lambda_q'(\beta f)=\beta^m\Lambda_q'f,\quad\mathrm{for}\ 
    f\in C^{s+1}((0,T)\times\partial\Omega)
    \ \mathrm{and}\ 
    \beta\in\mathbb{R}.
    % \ \mathrm{and}\ 
    % \beta\ge0.
\end{equation*}
%Furthermore, we define the \textit{nonlinear operator norm} of 
%$\Lambda_q'$ as
%\begin{equation*}
%    \|\Lambda_q'\|_\mathcal{N}:=\sup_{f\ne0}
%    \frac{\|\Lambda_q'f\|_{E^{s-\frac{1}{2}}((0,T)\times\partial\Omega)}}{\|f\|_{C^{s+1}((0,T)\times\partial\Omega)}^m}.
%\end{equation*}

\subsection{Main Result}
A natural obstruction for the proposed inverse coefficient problem is due to the finite speed of propagation. 
For instance, let us consider $m\geq 3$. In order to determine the value of $q(p)$ at a point 
$p\in(0,T)\times\Omega\subset\mathbb{R}^{1+n}$, 
there must exist a nontrivial solution to (\ref{2-1}) that 
propagates from $(0,T)\times\partial\Omega$ to $p$, and subsequently 
from $p$ back to $(0,T)\times\partial\Omega$. 
Due to the finite speed of wave propagation, a signal originating 
from a point $p'=(t_0,x_0)\in\mathbb{R}^{1+n}$ can influence only the region
\begin{equation*}
    \mathcal{J}_+(p')=\{(t,x)\in\mathbb{R}^{1+n}|t\ge t_0,\ |x-x_0|\le t-t_0\},
\end{equation*}
referred to as the $\mathbf{future}$ of $p'$. 
Similarly, we define the $\mathbf{past}$ of $p'$ by
\begin{equation*}
    \mathcal{J}_-(p')=\{(t,x)\in\mathbb{R}^{1+n}|t\le t_0,\ |x-x_0|\le t_0-t\},
\end{equation*}
We denote 
% the union of all such reachable sets from the boundary by
\begin{equation*}
    \mathcal{J}_\pm ((0,T)\times\partial\Omega)=\cup_{p'\in(0,T)\times\partial\Omega}
\mathcal{J}_\pm(p').
\end{equation*}
% Accordingly, the map $\Lambda_q$ contains no information about $q$ outside 
% the region
% \begin{equation}\label{regionD}
%     \mathbb{D}:=\mathcal{J}_+([d,T-d]\times\partial\Omega)\cap\mathcal{J}_-([d,T-d]\times\partial\Omega)\cap\left((0,T)\times\Omega\right),
% \end{equation}
% where $d>0$ is a small constant to ensure that the signal remains supported 
% within
% \begin{equation*}
%     \mathcal{J}_+((0,T)\times\partial\Omega)\cap\mathcal{J}_-((0,T)\times\partial\Omega)\cap\left((0,T)\times\Omega\right).
% \end{equation*}
% A schematic illustration of the region $\mathbb{D}$ is provided in Figure~\ref{fig:D}.

Accordingly, the map $\Lambda_q$ contains no information about $q$ outside 
the region
\begin{equation*}
    \mathcal{J}_+((0,T)\times\partial\Omega)\cap\mathcal{J}_-((0,T)\times\partial\Omega)\cap\left((0,T)\times\Omega\right),
\end{equation*}
and we define the region $\mathbb{D}$ as:
\begin{equation}\label{regionD}
    \mathbb{D}:=\mathcal{J}_+([d,T-d]\times\partial\Omega)\cap\mathcal{J}_-([d,T-d]\times\partial\Omega)\cap\left([d,T-d]\times\Omega\right),
\end{equation}
where $d>0$ is a small constant. A schematic illustration of the region $\mathbb{D}$ is provided in Figure~\ref{fig:D}.
Similarly for $m=2$, we assume $T\ge\mathrm{diam}(\Omega)+2d$ to ensure that every point $p\in\Omega$ can be detected by a signal originated from $(0,T)\times\partial\Omega$.

\begin{figure}[htbp]
    \centering
    \includegraphics[width=0.8\textwidth]{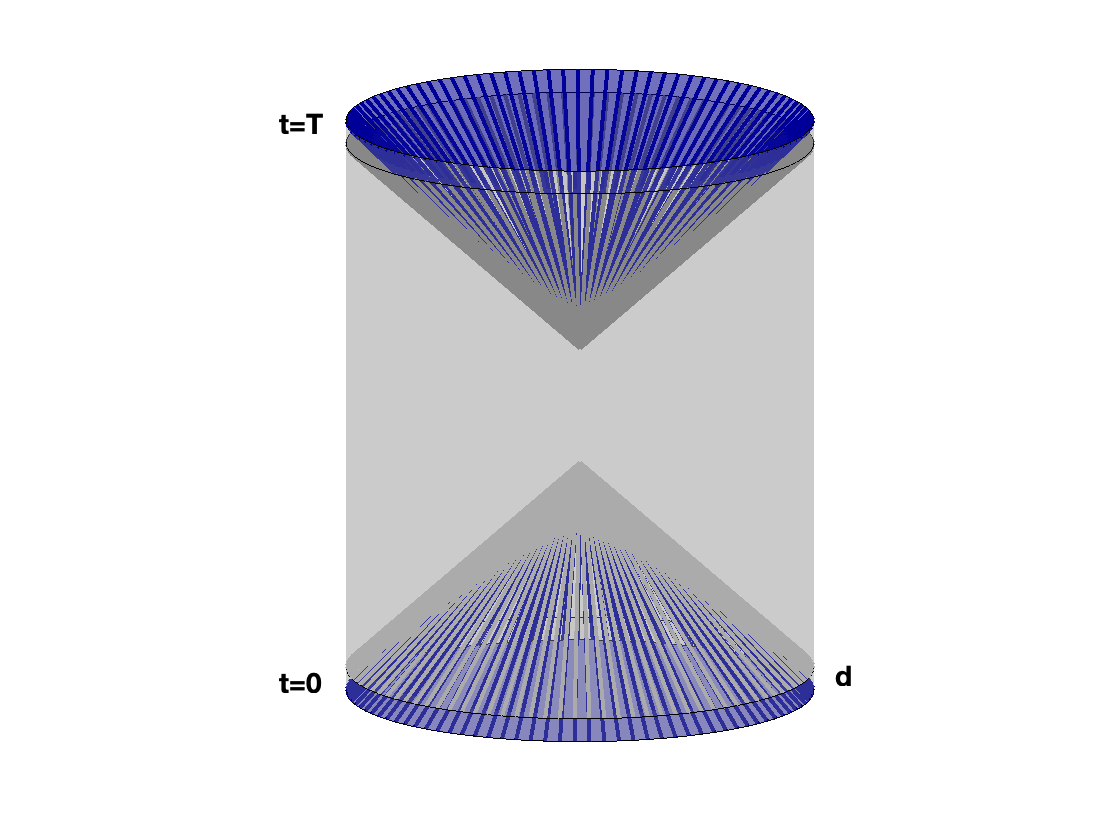}  % 插入图片，设置宽度为页面宽度的80%
    \caption{The grey region shows $\mathbb{D}$ in $(0,T)\times\mathbb{R}^2$.}  % 添加图片标题
    \label{fig:D}  % 给图片设置标签，用于引用
\end{figure}
The main result of our work is given as follows:

\begin{Thm}\label{Thm:main}
Let $\Omega$ be a bounded convex open domain in 
$\mathbb{R}^n$ $(n \geq 2)$ with $C^{\infty}$ boundary. Denote
$M>0$ and let $\|f\|_{C^{s+1}((0,T)\times\partial\Omega)}\leq M, s>(n+1)/2$.
Let the coefficient
\begin{equation*}
    q=\begin{cases}
        q(x)\in C^\infty(\overline{\Omega}),   & m=2, \\
        q(t,x)\in C^\infty([0,T]\times\overline{\Omega}), & m\ge3,   
    \end{cases}
\end{equation*}
% \textbf{[Q8: use $x$ instead of $\boldsymbol{x}$ to keep consistence.]}
in (\ref{2-1}) be sufficiently small in $H^s$-norm, and $T\ge\mathrm{diam}(\Omega)+2d$ when
$m=2$, where $d>0$ is given in (\ref{regionD}).
There holds the following stability estimates
\begin{align*}
    \left\{
    \begin{aligned}
        \|q\|_{L^\infty(\Omega)}&\le\mathcal{O}(\varepsilon^{\frac{1}{2(4s+n+11)}\cdot\frac{2s-n}{2s}}), & & m=2,\\
        \|q\|_{L^\infty(\mathbb{D})}&\le\mathcal{O}(
        \varepsilon^\frac{1}{4m(s+2)+n+8}), & & m\ge3,  
    \end{aligned}
    \right.
\end{align*}
where $\varepsilon := \sup_{\|f\|_{C^{s+1}((0,T)\times\partial\Omega)
}=1}\|\Lambda_q'(f)\|_{E^{s-\frac{1}{2}}((0,T)\times\partial\Omega)}$.
\end{Thm}

Below we discuss how our results differ from those in \cite{LassasP8}, which addresses a similar problem using higher-order linearization of the DtN map. More precisely, \cite{LassasP8} considers the inverse coefficient problem for the semilinear wave equation (\ref{2-1}) via higher-order linearizations of the DtN map, recovering the potential $q(t,x)$ in a compact set $W\subset[t_1,t_2]\times\Omega\subset[0,T]\times\Omega$. By employing the Radon transform, \cite{LassasP8} is able to recover both time- and space-dependent coefficients $q(t,x)$ for indices $m\geq 2$ and $n\geq 1$.
In contrast, our work considers the same inverse problem for (\ref{2-1}) but establishes a stability estimate through a first-order linearization of the NtD map. For \( m = 2 \), the reconstruction reduces to the X-ray transform, while for \( m \geq 3 \), we achieve pointwise reconstruction of the coefficient. By implementing the techniques developed in \cite{MattiLorentzian}, we are able to remove the compact support condition in \cite{LassasP8} and recover the coefficient function \( q(t,x) \) up to the boundary for \( m \geq 3 \) and \( q(x) \) for $m=2$.

The structure of this paper is as follows. Section \ref{se2} is dedicated to establishing the well-posedness of the forward problem, including the NtD map, the first-order linearization and introducing geometric optics solutions to the wave equation. In Section \ref{se3}, we derive the Alessandrini–PIE type identity, which relates the boundary observations to the unknown coefficient. Section \ref{se5} develops the stability result concerning the boundary of the spatial domain. Section \ref{se6} explains how the boundary terms of the geometric optics solutions should be chosen to construct appropriate boundary conditions. In Section \ref{se7}, we prove the stability estimate, separately treating the cases \(m \ge 3\) and \(m = 2\) using different approximation techniques. Finally, in Section \ref{se8}, 
numerical experiments demonstrate that the coefficient can be recovered from boundary measurements by a neural network inversion algorithm based on the least square framework.

%The structure of this paper is outlined as follows. Section 2 is 
%dedicated to establishing the well-posedness of the forward problem, 
%which includes the NtD map and the first order linearization 
%method. In Section 3, we derive the Alessandrini-PIE identity
%to establish the relationship between the boundary observations and the unknown
%interior information. Section 4 clarifies the region 
%where information about the interior medium can be 
%obtained and introduces the geometric optics solution 
%to the wave equation. 
%Section 5 develops the stability result for the boundary 
%of the spatial region. 
%Section 6 explains how the boundary terms of the 
%geometric optics should be chosen to
%construct detect functions. Finally, in Section 7, 
%we prove the stability estimate, dividing into the cases 
%$m\ge3$ and $m=2$, using different approximation kernels.

\section{Neumann-to-Dirichlet Map and Its Linearization}\label{se2}

% \textbf{[Q14: use $\overline{\Omega}$ instead of $\bar{\Omega}$.]}
% and we note 
% $\square u(t,x)=\partial_t^2 u(t,x)-\Delta 
% u(t,x)$. 
% In the case $m=2$, the only difference is that the coefficient depends 
% solely on space, 
% $q=q(x)\in H^s(\Omega)$; the proof can be adapted accordingly by 
% modifying the relevant norm.

In this section, we discuss the well-posedness of the semilinear wave equation (\ref{2-1}) and introduce the corresponding (linearized) NtD map. Furthermore, we construct geometric optics solutions for the wave equation and analyze their truncation properties.
\subsection{Well-Posedness and the Neumann-to-Dirichlet Map}
For sufficiently small $q=q(t,x)\in C^\infty([0,T]\times\overline{\Omega})$, we first establish the well-posedness of 
the NtD map $\Lambda_q$.
% \textbf{[for sufficiently small $q$. ]}
To this end, we introduce the energy space
\begin{equation*}
    E^s((0,T)\times\Omega):=\bigcap_{0\le k\le s} C^k((0,T);
    H^{s-k}(\Omega)),
\end{equation*}
which is equipped with the norm
\begin{equation*}
    \|u\|_{E^s}=\sup_{t\in(0,T)}\sum_{0\le k\le s} \|\partial_t^k 
                 u(t,\cdot)\|_{H^{s-k}(\Omega)}.
\end{equation*}

\begin{Thm}\label{Thm:1}
    Let $T>0$ be fixed. Assume that $f\in C^{s+1}((0,T)\times
    \partial\Omega)$ with $s>(n+1)/2$, and that 
    the compatibility conditions 
    \begin{equation*}
        \partial_t^l f|_{t=0}=0,\quad\mathrm{for\ 
    all}\ l=0,1,\dots,s-1,
    \end{equation*}
    are satisfied. 
    Fix $M>0$, for any $f$ with $||f||_{C^{s+1}}\le M$, there exists a constant
    $C_0(M)>0$,
    % \textbf{[$C$ depends on $M$ only or depends on $f$?]}, 
    such that if $\|q\|_{E^s}<C_0(M)$,
    the initial boundary value problem 
    (\ref{2-1}) admits a unique solution
    \begin{equation*}
        u\in {E^s((0,T)\times\Omega)},
    \end{equation*}
    which satisfies the estimate
    \begin{equation*}
        ||u||_{E^s((0,T)\times\Omega)} \le 
        C(m,M,T,\Omega)||f||_{C^{s+1}((0,T)\times\partial\Omega)},
    \end{equation*}
    where the constant $C(m,M,T,\Omega)>0$ is independent of $f$. 
    % \textbf{[Q17: double use of $C$ in different places. ]}
\end{Thm}
\begin{proof}
We fix $s>(n+1)/2$ to ensure that $E^s$ is an algebra, as shown in 
\cite{Lassas14,LassasP8}. 
Then, there exists a 
function $u^{(0)}\in E^{s} ((0,T)\times\Omega)$ satisfying the linear initial-boundary value problem (\ref{11-4}),
and it yields the energy estimate
\begin{equation*}
    ||u^{(0)}||_{E^{s} (0,T)
\times\Omega)}\le C||f||_{C^{s+1}((0,T)\times\partial\Omega)}\le 
CM,
\end{equation*} 
where $C=C(m,T,\Omega)>0$ is a constant.

Denote $\tilde{u}=u-u^{(0)}$, where $u$ solves the wave equation 
(\ref{2-1}). Then $\tilde{u}$ satisfies
\begin{equation}\label{11-6}
    \begin{cases}
        \square \tilde{u}= - q(t,x)(\tilde{u}+u^{(0)})^m, 
            &\mathrm{in}\ (0,T)\times\Omega, \\
        \partial_\nu \tilde{u} = 0, &\mathrm{on}\ (0,T)\times\partial\Omega, \\
        \tilde{u}(0,x)=\partial_t\tilde{u}(0,x)=0, &x\in\Omega.
    \end{cases}
\end{equation}
This equation is of the same form as in  
\cite[Eq. (5.12)]{Zhai8}. For $R>0$, we define the space $W(R,T)$ to consist of all 
functions $w$ such that
\begin{equation*}
    \begin{split}
       &  w \in \bigcap_{0\le k\le s} W^{k,\infty}((0,T);H^{s-k}(\Omega)), \\
       & ||w||_W :=||w||_{E^s}=\sup_{t\in(0,T)} \sum_{0\le k\le s} ||\partial_t^k 
                        w(t,\cdot)||_{H^{s-k}}\le R.
    \end{split}
\end{equation*}
We write the nonlinear term as
\begin{equation*}
    q(t,x)(\tilde{u}+u^{(0)})^m=
 q(t,x)(u^{(0)})^m+q(t,x)\left(\sum_{k=1}^{m}C_m^k
\tilde{u}^{k-1}(u^{(0)})^{m-k}\right) \tilde{u}, 
\end{equation*} 
%where
%\begin{align*}
%    H((t,x),u^{(0)}) &= q(t,x)(u^{(0)})^m,\\
%    G((t,x),\tilde{u},u^{(0)})&=-q(t,x)\sum_{k=1}^{m}C_m^k
%                            \tilde{u}^{k-1}(u^{(0)})^{m-k},
%\end{align*}
and obtain the estimate
\begin{equation*}
        \sup_{t\in(0,T)} \sum_{0\le k\le s} ||\partial_t^k (q(u^{(0)})^m)||_{H^{s-k}(\Omega)} 
        \le C(m,M,T,\Omega)||q||_{E^{s}((0,T)\times\Omega)}.
\end{equation*}
In addition, 
%it is known that $G((t,x),\tilde{u},u^{(0)})$ is essentially 
%a homogeneous polynomial of degree $m-1$ in terms of 
%$(\tilde{u},u^{(0)})$, thus 
we have
\begin{equation*}
    \begin{split}
        q\sum_{k=1}^{m}C_m^k
        \tilde{u}^{k-1}(u^{(0)})^{m-k} &\in \bigcap_{0\le k\le s} W^{k,\infty}
                                 ((0,T);H^{s-k}(\Omega)), \\
        \left\|q\sum_{k=1}^{m}C_m^k
        \tilde{u}^{k-1}(u^{(0)})^{m-k}\right\|_{E^s} &\le ||q||_{E^s}\cdot C 
            \sum_{k=0}^{m-1}||u^{(0)}||_{E^s}^{k} ||\tilde{u}||_W^{m-1-k}\\
        &\le C||q||_{E^s}(1+||\tilde{u}||_W),
    \end{split}
\end{equation*}
for $\tilde{u}\in W(R_0,T)$ with $R_0>0$ sufficiently small.

Given $\tilde{w}\in W(R_0,T)$, we consider the linear initial-boundary value problem:
\begin{equation}\label{11-7}
    \begin{cases}
        \square \tilde{u} 
            = - q(t,x)(\tilde{w}+u^{(0)})^m, &\mathrm{in}\ 
                (0,T)\times\Omega, \\
        \partial_\nu \tilde{u} 
            = 0, &\mathrm{on}\ (0,T)\times\partial\Omega, \\
        \tilde{u}(0,x)=\partial_t\tilde{u}(0,x)=0, &x\in\Omega.
    \end{cases}
\end{equation}
By Theorem 3.1 in \cite{Zhai8}, there exists a unique solution 
$\tilde{u}\in \bigcap_{0\le k\le s} W^{k,\infty}((0,T);H^{s-k}(\Omega))$ 
to (\ref{11-7}), and it satisfies the estimate
\begin{equation*}
    \begin{split}
        ||\tilde{u}||_{E^s} &\le CTe^{KT}\left(\left\|q\sum_{k=1}^{m}C_m^k
                                    \tilde{w}^{k-1}(u^{(0)})^{m-k}\right\|_{E^s}||\tilde{w}||_W+||q(u^{(0)})^m||_{E^s}\right) \\
        &\le CTe^{KT}||q||_{E^s} (1 +||\tilde{w}||_W),
    \end{split}
\end{equation*}
where $C,K>0$ are generic constants. Define $\mathcal{T}$ to be the map that maps 
$\tilde{w}\in W(R_0,T)$ to the solution $\tilde{u}$ of equation 
(\ref{11-7}). By choosing $R_0>0$ sufficiently small 
and ensuring that
$||q||_{E^s}<\frac{R_0e^{-KT}}{CT(1+R_0)}$, we can guarantee that 
\begin{equation*}
    ||\tilde{u}||_W<R_0,
\end{equation*}
so that $\mathcal{T}$ maps $W(R_0,T)$ into itself.

Suppose that there exist two solutions $\tilde{u}_1$ and $\tilde{u}_2$, each satisfying the equation
\begin{equation*}
    \begin{cases}
        \square \tilde{u}_j= - q(t,x)(\tilde{w}_j+u^{(0)})^m, 
            &\mathrm{in}\ (0,T)\times\Omega,\\
        \partial_\nu \tilde{u}_j = 0 &\mathrm{on}\ (0,T)\times\partial\Omega,\\
        \tilde{u}_j(0,x)=\partial_t\tilde{u}_j(0,x)=0, &x\in\Omega.
    \end{cases}
\end{equation*}
Assume further that $\tilde{u}_j=\mathcal{T}\tilde{w}_j$, $j=1,2$. 
Then the difference $\tilde{u}_1-\tilde{u}_2$ satisfies
\begin{equation*}
    \begin{split}
        \square(\tilde{u}_1-\tilde{u}_2) &= - q(t,x)(\tilde{w}_1+u^{(0)})^m +  q(t,x)(\tilde{w}_2+u^{(0)})^m \\
        &:= q(t,x)\tilde{G}((t,x),\tilde{w}_1,\tilde{w}_2,u^{(0)})(\tilde{w}_1-\tilde{w}_2),
    \end{split}
\end{equation*}
where $\tilde{G}((t,x),\tilde{w}_1,\tilde{w}_2,u^{(0)})$ 
is a polynomial function depending on 
$\tilde{w}_1(t,x)$, $\tilde{w}_2(t,x)$ and $u^{(0)}(t,x)$.
% \textbf{[Q15: not a polynomial in $(t,x)$]}
It then follows that
\begin{equation*}
    ||\mathcal{T}\tilde{w}_1-\mathcal{T}\tilde{w}_2||_{E^s} = 
    ||\tilde{u}_1-\tilde{u}_2||_{E^s} \le C||q||_{E^s} ||\tilde{w}_1-\tilde{w}_2||_{E^s} e^{KT}.
\end{equation*}
% \textbf{[Q16: use $\|\cdot\|_{E^s}$ not $\|\cdot\|_W$ if you do not know aprior it is in $W$.]}
By choosing $||q||_{E^s}$ sufficiently small such that 
$C||q||_{E^s} e^{KT}<1$, the map $\mathcal{T}$ becomes a 
contraction. Consequently, the equation (\ref{11-6}) 
admits a unique solution 
$\tilde{u}\in W(R_0,T)$. Applying Theorem 3.1 in \cite{Zhai8} 
once again, we conclude that
$$
    \tilde{u}\in {E^s((0,T)\times\Omega)}.
$$
\end{proof}

Based on the well-posedness of the initial-boundary value problem (\ref{2-1}) established above, we can then define the corresponding NtD map $\Lambda_q$. In particular, this definition leads to the following linearized NtD map, which will be used in the present work to establish the stability estimates.

\subsection{The Linearized NtD Map}
We recall the linearized NtD map  
\begin{equation*}  
    \Lambda_q' : C^{s+1}((0,T)\times\partial\Omega) \to E^{s-\frac{1}{2}}((0,T)\times\partial\Omega)  
\end{equation*}  
defined in (\ref{11-3}). In the following proposition, we show that it indeed corresponds to the Fréchet derivative with respect to the coefficient \(q\).
\begin{Prop}\label{Prop:1}
    Let assumptions of Theorem \ref{Thm:1} hold. Then, 
    for any $f\in C^{s+1}((0,T)\times\partial\Omega)$ with 
    $||f||_{C^{s+1}}\le M$, the following estimate holds:
    \begin{equation*}
        ||\Lambda_qf-\Lambda_0f-\Lambda_q'f||_{E^{s-\frac{1}{2}}
        ((0,T)\times\partial\Omega)} \le C(T,\Omega,M) 
        ||q||_{E^s((0,T)\times\Omega)}^2.
    \end{equation*}
\end{Prop}
\begin{proof}
By the definitions of the NtD map $\Lambda_q$ in (\ref{eq_NtD map}) and the 
linearized NtD map $\Lambda_q'$ in (\ref{11-3}), we have
\begin{equation*}
    \Lambda_q f - \Lambda_0 f - \Lambda_q'f = (u-u^{(0)}-u^{(1)})|_{(0,T)\times\partial\Omega}, \qquad \mathrm{for}\ ||f||_{C^{s+1}((0,T)\times\partial\Omega)}\le M,
\end{equation*}
where $u$, $u^{(0)}$, and $u^{(1)}$ solve the original problem in 
(\ref{2-1}), the unperturbed problem $(I_0)$ in (\ref{11-4}) and the 
linearized problem $(I_1)$ in (\ref{11-5}), respectively.

We adopt all the assumptions stated in Theorem \ref{Thm:1} and notations in 
Definition \ref{Def:1}. By subtracting the original problem given in equation 
(\ref{2-1}) from the unperturbed problem denoted as $(I_0)$ in (\ref{11-4}), we obtain
\begin{equation*}
    \begin{cases}
        \square (u-u^{(0)}) = -qu^m, & \mathrm{in}\ (0,T)\times\Omega, \\
        \partial_\nu (u-u^{(0)}) = 0, & \mathrm{on}\ (0,T)\times\partial\Omega, \\
        (u-u^{(0)})(0,x)=\partial_t (u-u^{(0)})(0,x)=0, & x\in\Omega.
    \end{cases}
\end{equation*}
This leads to the estimate
\begin{equation*}
        ||u-u^{(0)}||_{E^s((0,T)\times\Omega)} 
           \le C(T,\Omega)||q||_{E^s((0,T)\times\Omega)}||u||_{E^s((0,T)\times\Omega)}^m.
\end{equation*}
% \textbf{[Why $\|\cdot\|_{E^{s-2}}$?]}
Similarly, we have
\begin{equation*}
    ||u-u^{(0)}-u^{(1)}||_{E^s((0,T)\times\Omega)} \le C(T,\Omega)||q||_{E^s((0,T)\times\Omega)}||u^m-(u^{(0)})^m||_{E^s((0,T)\times\Omega)}.
\end{equation*}
By the trace theorem, 
\begin{equation*}
    ||u||_{E^{s-\frac{1}{2}}((0,T)\times\partial\Omega)}
\le C(T,\Omega)||u||_{E^s((0,T)\times\Omega)},
\end{equation*}
and hence,
\begin{equation*}
    ||\Lambda_qf-\Lambda_0f-\Lambda_q'f||_{E^{s-\frac{1}{2}}
    ((0,T)\times\partial\Omega)} \le C(T,\Omega)||q||_{E^s((0,T)\times\Omega)}
    ||u^m-(u^{(0)})^m||_{E^s((0,T)\times\Omega)}.
\end{equation*}

By Theorem \ref{Def:1} and the preceding estimate, the final term 
can be bounded as follows:
\begin{equation*}
    \begin{split}
        ||u^m-(u^{(0)})^m||_{E^s((0,T)\times\Omega)} 
           &\le C(T,\Omega)||u-u^{(0)}||_{E^s((0,T)\times\Omega)} \sum_{j=1}^m ||u||_{E^s((0,T)\times\Omega)}^{m-j} ||u^{(0)}||_{E^s((0,T)\times\Omega)}^{j-1}\\
           &\le C(T,\Omega)||q||_{E^s((0,T)\times\Omega)}\sum_{j=1}^m ||u||_{E^s((0,T)\times\Omega)}^{2m-j} ||u^{(0)}||_{E^s((0,T)\times\Omega)}^{j-1}\\
           &\le C(T,\Omega)||q||_{E^s((0,T)\times\Omega)}\sum_{j=1}^m ||f||_{C^{s+1}((0,T)\times\partial\Omega)}^{2m-1} \\
           &\le C(T,\Omega) M^{2m-1}||q||_{E^s((0,T)\times\Omega)}.
    \end{split}
\end{equation*}
Hence, we obtain
\begin{equation*}
    ||\Lambda_qf-\Lambda_0f-\Lambda_q'f||_{E^{s-\frac{1}{2}}
    ((0,T)\times\partial\Omega)} \le C(T,\Omega) M^{2m-1}
    ||q||_{E^s((0,T)\times\Omega)}^2.
\end{equation*}
\end{proof}

\subsection{Geometric Optics Solutions}\label{se4}
To further investigate the stability of the inverse coefficient problem for the semilinear wave equation using the linearized NtD map, it is necessary to construct solutions to the linear wave equation $\square u = 0$, for instance in (\ref{11-4}). In particular, we employ the conventional method of asymptotic geometric optics, and denote the resulting solution to $\square u = 0$ as \( u_\tau(y, \xi) \).
Let $y:=(t,x)$, following \cite{AliLauri}, 
the construction is based on the ansatz
\begin{equation}\label{12-7}
    u_{\tau}(y,\xi)=e^{i\tau\xi\cdot y}a_{\tau}(y) = e^{i\tau\xi\cdot y}\left( \sum_{k=0}^N \frac{a_{k}(y)}{\tau^k}\right),
\end{equation}
where $\tau>0$ is a large parameter, and 
$N>0$ is a fixed integer large enough. 
The covector 
$\xi=(\xi_{0},\xi')=(\xi_{0},\xi_{1},...,\xi_{n})$ 
is required to be light-like, i.e., 
$|\xi_{0}|^2=|\xi'|^2:=|\xi_{1}|^2+\cdots+|\xi_{n}|^2$.
The inner product is defined by $\xi\cdot y:=\xi_0t+\sum_{i=1}^n\xi_ix_i$.
We emphasize the uniformity of this construction, as we 
aim to keep track of the constants involved, which are essential for 
the subsequent stability estimates.

We denote by $\xi^\sharp:=(-\xi_{0},\xi')$ the vector associated with 
$\xi$ under the Minkowski metric. Let 
$y^0=(t^0,x^0)\in\mathbb{R}^{1+n}$. The 
amplitudes $a_{k}$, for $k=0,1,...,N$, will be constructed such that 
$u_{\tau}$ satisfies the wave equation  $\square u = 0$ module a remainder term 
vanishing as $\tau\to\infty$. The ansatz is supported in a neighborhood 
of the line 
\begin{equation}\label{ex12}
    \gamma_{y^0,\xi}(s):=s\xi^\sharp+y^0=(-s\xi_{0}+t^0,s\xi'+x^0), 
    \quad\forall s\in\mathbb{R}.
\end{equation}

Since $\xi$ is light-like, the wave operator acting on the ansatz yields
\begin{equation}\label{ex13}
    \square(e^{i\tau\xi\cdot y}a_{\tau})
    =e^{i\tau\xi\cdot y}(-2i\tau\mathcal{T}_{\xi}a_{\tau}+\square a_{\tau}),
\end{equation}
where the transport operator is defined as
$\mathcal{T}_{\xi}:=-\xi_{0}\partial_{t}+\sum_{l=1}^{n}\xi_{l}\partial_{x_{l}}$. 
The amplitudes $a_{k}$ are constructed such that
% is guided by the requirement 
% that the right-hand side of (\ref{ex13}) vanishes in powers of $\tau$. 
% In particular, this leads to the transport equation
\begin{equation}\label{ex14}
    \mathcal{T}_{\xi}a_{0}=0,
\end{equation}
and
\begin{equation}\label{ex18}
    -2i\mathcal{T}_{\xi}a_{k}+\square a_{k-1}=0,\quad k\ge1.
\end{equation}
Moreover, if a vector $\omega\in\mathbb{R}^{1+n}$ satisfies 
$\xi^\sharp\cdot\omega=0$, then for any function $\chi\in C^1(\mathbb{R})$, 
we have
\begin{equation*}
    \mathcal{T}_{\xi}(\chi(\omega\cdot(y-y^0)))=0.
\end{equation*}

We select 
\begin{equation}\label{ex15}
    \frac{\xi'}{|\xi'|}, w_{1}', w_{2}',..., w_{n-1}' 
\end{equation}
to form an orthonormal basis of $\mathbb{R}^{n}$ with respect to 
the Euclidean metric, and define $w_{l}:= (0,w_{l}')$. Note that 
$\xi^\sharp\cdot\omega_{l}=0$, and
\begin{equation*}
    \{\gamma_{y^0,\xi}(s)|s\in\mathbb{R}\} = 
    \{y\in\mathbb{R}^{1+n}| \xi\cdot(y-y^0)=\omega_{1}\cdot(y-y^0)=
       \cdots = \omega_{n-1}\cdot(y-y^0)=0\}.
\end{equation*}
Let $\delta>0$, and take $\chi_\delta\in C_c^\infty((-\delta,\delta))$. 
We define the first order term of the amplitudes as
\begin{equation}\label{aj0}
    a_{0}(y) = \chi_\delta(\xi\cdot(y-y^0))
    \prod_{l=1}^{n-1}\chi_\delta(w_{l}\cdot(y-y^0)).
\end{equation}
Then $a_{0}$ satisfies the transport equation (\ref{ex14}), and we have 
the support inclusion
\begin{equation}\label{ex17}
    \mathrm{supp}(a_{0}(t,\cdot))\subset H(t,\delta),\quad\forall t\in\mathbb{R},
\end{equation}
where $H(t,\delta)\subset\mathbb{R}^n$ denotes the hypercube of side 
length $2\delta$, centered at the point $x\in\mathbb{R}^n$ such that 
$(t,x)=\gamma_{y^0,\xi}(s)$ for some $s\in\mathbb{R}$, and with edges 
aligned with the directions as $\xi^\sharp$. 
% \textbf{[Q19: should not be aligned with $\xi'$.]} 
The subsequent terms $a_{k}$ for $k\ge1$ 
are defined inductively by solving the transport equations (\ref{ex18}).
% \begin{equation}\label{ex18}
%     -2i\mathcal{T}_{\xi}a_{k}+\square a_{k-1}=0.
% \end{equation}
Imposing vanishing initial conditions on the hyperplane
$\Sigma_{y^0,\xi}=\{y\in\mathbb{R}^{1+n}|\xi^\sharp\cdot(y-y^0)=0\}$,
we obtain the solution 
\begin{equation}\label{12-9} 
    a_{k}(s\xi^\sharp+y) =\frac{1}{2i}\int_0^s
    (\square a_{k-1})(\tilde{s}\xi^\sharp+y)
    d\tilde{s},
\end{equation}
for $s\in\mathbb{R}$, $y\in\Sigma_{y^0,\xi}$. By induction from 
(\ref{ex17}), it follows that
$$
\mathrm{supp}(a_{k}(t,\cdot))\subset 
H(t,\delta),
$$
and thus the entire ansatz $u_{\tau}$ remains localized near 
$\gamma_{y^0,\xi}$. 

For convenience in the construction, we define the smooth cutoff function
\begin{equation*}
    \begin{split}
        \psi(y):=\begin{cases}
            e^\frac{1}{|y|^2-1}, & |y|<1,\\
            0, &|y|\ge1,
        \end{cases}
    \end{split}
\end{equation*}
and normalize it by setting
$\alpha(y):=\frac{\psi(y)}{\int_{\mathbb{R}^{1+n}}\psi(y)dy}$.
Then, for $\delta>0$, we define the scaled cutoff function as
\begin{equation}\label{cutoff}
    \chi_\delta(y):=\alpha \left(\frac{y}{\delta}\right). 
\end{equation}
This completes the construction of the geometric optics solutions 
$u_{\tau}(y)$.% Emphasizing the dependence of 
%(\ref{ex24-1}) on $y_j$ and $\xi_j$, we write
%\begin{equation}\label{ex24}
%    f_{y_j,\xi_j}=f_j.
%\end{equation}

Moreover, equations (\ref{12-7}) and (\ref{ex13}) imply
\begin{equation*}
        \square u_{\tau} 
            = e^{i\tau\xi\cdot y}\left( -2i\tau\mathcal{T}_{\xi}a_{0}(y)
                +\sum_{k=1}^{N}\frac{-2i\mathcal{T}_{\xi}a_{k}(y)+\square a_{k-1}(y)}{\tau^{k-1}}
                + \frac{\square a_{N}(y)}{\tau^N} \right).
\end{equation*}
In view of the transport equations (\ref{ex14}) and (\ref{ex18}), we 
conclude that 
\begin{equation*}
    \square u_{\tau} = e^{i\tau\xi\cdot y}
        \frac{\square a_{N}(y)}{\tau^N}.
\end{equation*}

Next, recalling the geometric optics ansatz
(\ref{12-7}) and the definition of $a_{0}$ in (\ref{aj0}),
we obtain the estimate:
\begin{equation*}
    \left|\frac{d}{dy} a_{0}(y)\right|
    \le \bar{C}|\xi| \frac{1}{\delta},
\end{equation*}
where the constant $\bar{C}$ depends only on 
the profile of the cutoff function $\chi_1$. Moreover, $\bar{C}$
admits uniform positive lower and upper bounds independent of all 
$y^0\in(0,T)\times\overline{\Omega}$ 
and all light-like vectors $\xi\in \{(\xi_0,\xi')\ |\ \xi'\in \mathbb{R}^{n},\ |\xi_0|=|\xi'|\}$. 
We use $\bar{C}$ to emphasize this uniformity property.
Unless otherwise stated, we will continue to use this notation to 
indicate that such constants admit uniform positive lower and upper 
bounds.

For $k>0$, applying the coordinate transform 
$y=s\xi^\sharp+z$ to (\ref{12-9}), we obtain
\begin{equation*}
    a_{k}(y)=a_{k}(y,z) = a_{k}(s\xi^\sharp+z,z) 
    = \frac{1}{2i} \int_0^\frac{|y-z|}{|\xi^\sharp|} 
    (\square a_{k-1}) (\tilde{s}\xi^\sharp+z) d\tilde{s},
\end{equation*}
where $|\xi^\sharp|$ denotes the Euclidean norm of $\xi^\sharp$.
It follows that
\begin{equation*}
    \left|\frac{d}{dy} a_{k}(y)\right|
    \le \bar{C}\frac{1}{|\xi|}\left| \square a_{k-1}(y)\right| 
    \le \bar{C}\frac{1}{|\xi|^k}\left|\frac{d^{k+1}}{dy^{k+1}} a_{0}(y)\right| 
    \le\bar{C}|\xi| \frac{1}{\delta^{k+1}}.
\end{equation*}
Similarly, for any $l\in\mathbb{N}^+$, we have
\begin{equation}\label{ajkl}
    \left|\frac{d^l}{dy^l} a_{k}(y)\right|
    \le\bar{C}|\xi|^l \frac{1}{\delta^{k+l}}.
\end{equation}
We now provide the following estimate
\begin{equation*}
    \begin{split}
        \left|\frac{d^k}{dy^k}(\square u_{\tau})\right| &= 
            \left|\frac{1}{\tau^N}
            \sum_{l=0}^{k}(e^{i\tau\xi\cdot y})^{(l)}
            (\square a_{N}(y))^{(k-l)}\right| \\
          &\le \bar{C}|\xi|^{k+2} \frac{1}{\tau^N}
            \sum_{l=0}^{k}\tau^l \delta^{-N-k-2+l} \\
          &\le \bar{C}|\xi|^{k+2} \delta^{-N-2}\tau^{-N+k},
    \end{split}
\end{equation*}
where we have used the assumption $\tau\delta>1$. 
Therefore, the error estimates for the geometric optics solutions are
\begin{equation}\label{12-8}
\begin{split}
    ||\square u_{\tau}||_{C^k((0,T)\times\mathbb{R}^n)}
        &\le\bar{C}|\xi|^{k+2}\delta^{-N-2}\tau^{-N+k}.
\end{split}
\end{equation}

\section{The Alessandrini-PIE Type Identity}\label{se3}

Following the well-posedness of both the 
NtD map and its first-order linearization, a standard follow-up approach 
for inverse coefficient problems involves deriving the 
so-called Alessandrini-type identity, analogous to the one introduced for inverse coefficient problems of elliptic equations in \cite{Aless1}. 
This identity establishes a crucial link between boundary measurements and the unknown interior properties of the coefficient.

% For the problem considered in this work, the primary 
% challenge arises from the nonlinearity present in the NtD map, 
% which will be addressed in the following subsection.

% In \cite{LassasP8}, the Alessandrini-type identity is 
% derived via the 
% higher order linearization of the small boundary data 
% for the wave equation. In contrast, we derive the identity 
% here using the first order 
% linearization of the small coefficient function. 
We introduce the auxiliary (or test) function $\varphi$, 
which satisfies 
\begin{equation}\label{f0}
    \begin{cases}
        \square \varphi(t,x) = 0, & \mathrm{in}\ (0,T)\times\Omega, \\
        \partial_\nu \varphi(t,x) = f_0, & \mathrm{on}\ (0,T)\times\partial\Omega, \\
        \varphi(T,x)=\partial_t \varphi(T,x)=0, & x\in\Omega,
    \end{cases}
\end{equation}
with its boundary value $f_0$ to be determined in the subsequent proof. 

Multiplying equation (\ref{11-5}) by the 
test function $\varphi$ and integrating over $(0,T)\times\Omega$, 
and noting that 
$\partial_\nu u^{(1)}|_{(0,T)\times\partial\Omega}=0$, we obtain 
the Alessandrini-type identity for the linearized NtD map 
$\Lambda_q'$:
\begin{equation}\label{2.12}
    \int_0^T\int_\Omega q(u^{(0)})^m\varphi dxdt
    =-\int_0^T\int_{\partial\Omega}\partial_\nu\varphi\Lambda_q'fdSdt.
\end{equation}

We emphasize that in the case of nonlinearity with \( m \ge 2 \), whether the Alessandrini type identity alone is sufficient to prove stability remains unclear to be demonstrated in \cite{LuSaloXu} for the inverse coefficient problem associated with Helmholtz-type semilinear equations. To further refine the above equality, we employ the Principle of Inclusion-Exclusion (PIE) from combinatorics, which has proven effective for the inverse coefficient problem of Helmholtz-type semilinear equations with a general nonlinearity index \( m \geq 2 \) in \cite{Zou}. In particular, the following PIE equality is essential.

%
%
%multiple detect functions $u_j^{(0)}$ are required to 
%derive a stability estimate, which is essential for 
%stable reconstruction. Motivated 
%by the strategy for $m=2$ presented in Theorem 3.1 of \cite{LuSaloXu}, 
%we establish 
%an Alessandrini-type identity involving multiple detect 
%functions for the general case
%$m\ge2$, which plays a central role in our analysis. The key 
%ingredient 
%is the following identity, derived from the Principle of 
%Inclusion-Exclusion (PIE) in combinatorics \cite{Zou}.

\begin{Lem}\label{Lem:1}\cite{Zou}
Let $m\in\mathbb{N}^+$, and define the index set $U:=\{1,2,\dots,m\}$. Let $S$ be a subset of $U$. Then, for any given set of numbers or functions $\{w_j\}_{j=1}^m$, the following identity holds:
    \begin{equation}\label{2.13}
        \prod_{j\in U}w_j=\frac{1}{m!}\sum_{\emptyset\subsetneqq S\subseteq U }(-1)^{|U\setminus S|}\left(\sum_{j\in S}w_j\right)^m.
    \end{equation}
    Here, $U\setminus S$ is the relative complement of $S$ in $U$, 
    and $|\cdot|$ is the cardinality of a set.
\end{Lem}

Similar to the approach of \cite{Zou} for Helmholtz-type semilinear equations, we therefore derive the following Alessandrini-PIE type identity for our semilinear wave equation (\ref{2-1}).
\begin{Prop}\label{Prop:2}
    Let $m\in\mathbb{N}$, $m\ge2$, define the index sets 
    $U:=\{ 1,2,...,m\}$ and $\emptyset\ne S\subseteq U$. The auxiliary 
    function $\varphi$ is defined in (\ref{f0}). 
    Let $u_j^{(0)}$ be the solution to (\ref{11-4}) with the 
    boundary value $\partial_\nu u_j^{(0)}=f_j$, 
    $(t,x)\in(0,T)\times\partial\Omega$, and $u_S^{(0)}$ be the solution to 
    (\ref{11-4}) with the boundary value $\partial_\nu u_S^{(0)}= \sum_{j\in S} f_j$, $(t,x)\in(0,T)\times\partial\Omega$, then $u_S^{(0)}=\sum_{j\in S} u_j^{(0)}$, $(t,x)\in(0,T)\times\Omega$, $u_S^{(1)}$ is the solution to
    \begin{equation*}
        (I_S^{(1)})\begin{cases}
            \square u_S^{(1)} + q (u_S^{(0)})^{m} = 0, 
            & \mathrm{in}\ (0,T)\times\Omega, \\
        \partial_\nu u_S^{(1)} = 0, & \mathrm{on}\ (0,T)\times\partial\Omega ,\\
        u_S^{(1)}(0,x)=\partial_t u_S^{(1)}(0,x)=0, & x\in\Omega.
        \end{cases}
    \end{equation*}
    Given $m$ detect functions $\left\{ f_j = \partial_\nu 
    u_j^{(0)} |_{(0,T)\times\partial\Omega} \right\}_{j=1}^{m}$, 
    it holds that
    \begin{equation}\label{12-6}
        \int_{(0,T)\times\Omega} q \prod_{j\in U} u_j^{(0)} \varphi dxdt = -\frac{1}{m!}\sum_{S\subset U} (-1)^{|U\setminus S|} \int_{(0,T)\times\partial\Omega} \Lambda_q' \left( \sum_{j\in S} f_j\right) \partial_\nu \varphi dS_xdt.
    \end{equation}
\end{Prop}
\begin{proof}
In view of Lemma \ref{Lem:1}, we denote by \( u_j^{(0)} \) the solution of the wave equation \( (I_0) \) in \eqref{11-4} satisfying the Neumann boundary condition
\begin{equation*}
    \partial_\nu u_j^{(0)}=f_j\quad\mathrm{on}\ (0,T)\times\partial\Omega
\end{equation*}
for each $j\in U=\{1,2,...,m\}$, and denote $u_S^{(0)}$ as the 
combined solution of wave equation ($I_0$) in (\ref{11-4}) with 
Neumann boundary condition
\begin{equation*}
    \partial_\nu u_S^{(0)}=\sum_{j\in S}f_j\quad\mathrm{on}\ (0,T)\times\partial\Omega
\end{equation*}
for each non-empty subset $S\subseteq U$. By the linearity of 
wave equation (\ref{11-4}), we have
\begin{equation*}
    u_S^{(0)}=\sum_{j\in S}u_j^{(0)}\quad\mathrm{in}\ (0,T)\times\Omega.
\end{equation*}
Thus, by Lemma \ref{Lem:1}, we obtain
\begin{equation}\label{2.15}
    \begin{split}
    \int_{(0,T)\times\Omega}q\prod_{j\in U}u_j^{(0)}\varphi dxdt
    &=\frac{1}{m!}\sum_{\emptyset\subsetneqq S\subseteq U }(-1)^{|U\setminus S|}\int_{(0,T)\times\Omega}q\left(u_S^{(0)}\right)^m\varphi dxdt.
    \end{split}
\end{equation}

Therefore, denote $u_S^{(1)}$ as the corresponding solution of the 
linearized equation
\begin{equation}\label{2.16}
    (I_S^{(1)})\begin{cases}
        \square u_S^{(1)} + q(u_S^{(0)})^m = 0, & \mathrm{in}\ (0,T)\times\Omega, \\
        \partial_\nu u_S^{(1)} = 0, & \mathrm{on}\ (0,T)\times\partial\Omega, \\
        u_S^{(1)}(0,x)=\partial_t u_S^{(1)}(0,x)=0, & x\in\Omega.
    \end{cases}
\end{equation}
Then, for each $S$, we obtain
\begin{equation}\label{2.17}
    \int_0^T\int_\Omega q(u_S^{(0)})^m\varphi dxdt
    =-\int_0^T\int_{\partial\Omega}u_S^{(1)}\partial_\nu\varphi dSdt,
\end{equation}
by the Alessandrini type identity (\ref{2.12}). Following the 
definition of the linearized NtD map $\Lambda_q'$, this means that, 
for each non-empty subset $S\subseteq U$,
\begin{equation}\label{2.18}
    u_S^{(1)}|_{(0,T)\times\partial\Omega}
    =\Lambda_q'(\partial_\nu u_S^{(0)}|_{(0,T)\times\partial\Omega})
    =\Lambda_q'\left(\sum_{j\in S}f_j\right).
\end{equation}

Combining (\ref{2.15}), (\ref{2.17}) and (\ref{2.18}), we have 
the Alessandrini-PIE type identity, which is a combinatoric formula of 
first-order linearization.
\end{proof}

\section{Boundary Stability}\label{se5}
As emphasized in the introduction, we do not assume compactness of the unknown coefficient \( q \). In this section, we specifically verify the stability of recovering the coefficient \( q \) at the boundary. 
Manifolds are introduced here for local flattening to obtain a simpler expression, facilitating the establishment of boundary stability. Only the results are needed to be used in Euclidean coordinates later.

Let \( y_0 \in [d,T-d] \times \Gamma := [d,T-d] \times \partial\Omega \).  
Denote by \( \mathcal{U} \) a sufficiently small neighborhood of \( y_0 \) in \( \mathbb{R}^{1+n} \),  
and let \( U := \mathcal{U} \cap [d,T-d] \times \Gamma \bigr] \).  
Assume \( (x^0, x^1, \dots, x^{n-1}) \) are local boundary coordinates on \( (0,T) \times \Gamma \) near \( y_0 \).  
We shall work in the boundary normal coordinates described in the lemma below; see \cite{StefanovYang} for further details.

\begin{Lem}\cite{StefanovYang}\label{Lem:3}
    There exists a neighborhood $N$ of $y_0$ in $(0,T)\times\Omega$, and a 
    diffeomorphism $\Psi$: $[(0,T)\times\Gamma]\cap N\times[0,T)\to(0,T)\times\overline{\Omega}$ such that
    \begin{itemize}
        \item $\Psi(x',0)=x'$ for all $x'\in[(0,T)\times\Gamma]\cap N$;
        \item $\Psi(x',x^n)=x'-x^n\nu$, where $\nu$ is the unit outer normal at the boundary point $x'$.
    \end{itemize}
\end{Lem}

The coordinate system $(t=x^0,x^1,...,x^{n-1},\tau=x^n)$ gives us a way to define the boundary normal coordinates locally, and the 
    metric tensor $g$ takes the form
    \begin{equation*}
        g=-dx_0dx_0+g_{\alpha\beta}dx^\alpha dx^\beta+d\tau d\tau,
    \end{equation*}
    where $\alpha,\beta=1,...,n-1$.

\begin{Rmk}
    We use the same notation as in \cite{StefanovYang} for consistency. The only difference is that, in our setting, the first coordinate already represents the time variable. Therefore, unlike in \cite{StefanovYang}, the diffeomorphism is chosen to modify only the spatial coordinates while leaving the time coordinate unchanged.
\end{Rmk}

We denote the wave operator by $\square_g$ in local coordinates 
$y=(x^0,x^1,...,x^n)$, which takes the form
\begin{equation*}
    \square_g:=\frac{1}{\sqrt{|\mathrm{det}g|}}\left[-\partial_{0}(\sqrt{|\mathrm{det}g|}\partial_{0})+
        \partial_j(\sqrt{|\mathrm{det}g|}g^{jk}\partial_k)
        +\partial_n(\sqrt{|\mathrm{det}g|}\partial_n)
        \right],
\end{equation*}
where we have $(g^{jk})=(g_{\alpha\beta})^{-1}$, $j,k=1,...,n-1$. 
The boundary stability of the recovery of $q$ is given by the following theorem:

%\fix{Should we fix $y_0$ or $x_0$ to be the boundary point?  }
%\fix{Fixed.}
\begin{Thm}\label{Thm:2}
Let $y_0$ be a fixed point on the boundary $[d,T-d]\times\partial\Omega$. 
Given any \(0<\mu<1\), there exists a neighborhood \(U_0\subset (0,T)\times\partial\Omega\) of \(y_0\) such that the estimate
\[
\sup_{y\in U_0,\,|\gamma|\le 1} |\partial^\gamma q(y)|\le C\varepsilon^{{\frac{\mu}{4}\frac{1}{m(s+2)+2}}}
\]  
holds whenever \(q\) is bounded in a suitable \(C^k\)-norm near \(y_0\).  
Here the constant \(C>0\) depends on $m$ and the norm bound, and the required regularity order \(k\) is determined by \(k=k(\mu)\).
\end{Thm}
\begin{proof}
We choose the local coordinates %$(x',x^n)$ 
so that $(0,T)\times\Gamma$ locally is given by $x^n=0$, and 
the interior of $\Omega$ is given by $x^n>0$. 
Denote $y=(t,x^1,...,x^n)=(x^0,x^1,...,x^n)$, and choose
\begin{equation*}
    u_\lambda^{(0)}(y)=e^{i\lambda\phi(y,\xi')}a_\lambda(y,\xi')=e^{i\lambda\phi(y,\xi')}\sum_{j=0}^{N}\frac{1}{\lambda^j}a_{0j}(y,\xi')
\end{equation*} 
as the asymptotic solution to $\square u^{(0)}=0$, 
where $\xi=(\xi',\xi_n)=(-1,\xi_1,...,\xi_{n-1},\xi_{n})$.
% , and therefore $\phi|_{x^n=0}=-t+\varphi(x')=x\cdot\xi$. 
In $(0,T)\times\Omega$ near $y_0$, the phase function $\phi(y,\xi')$ satisfies the 
Eikonal equation, which, under the boundary normal coordinates,
can be written as
\begin{equation}\label{eikonal}
    -(\partial_0\phi)^2+g^{\alpha\beta}\partial_\alpha\phi\partial_\beta\phi+(\partial_n\phi)^2=0,%\quad\alpha,\beta=0,1,...,n-1,
\end{equation}
and when restricted on the boundary, can be written as
\begin{equation}\label{eikoCon}
    \phi|_{x^n=0}=y\cdot\xi|_{x^n=0}=x'\cdot\xi'=-x^0+x^1\xi_1+\cdots+x^{n-1}\xi_{n-1}.
\end{equation}
With the extra condition $\partial_\nu\phi|_\Gamma<0$, the equation
(\ref{eikonal}) is uniquely solvable with the boundary condition (\ref{eikoCon}). 
% \textbf{[Q22: the construction of the geometric optics is only local. Also what is $a_{0j}$? ]}
Moreover, we know that 
\begin{equation*}
    \partial_n\phi(x',0)=\xi_n(x',\xi')>0,
\end{equation*}
for any $(x',x^n)$ in a neighborhood of $y_0$, so
\begin{equation*}
    \xi_n(x',\xi')=\partial_n\phi(x',0)=\sqrt{1-g^{\alpha\beta}\xi_\alpha\xi_\beta}.
\end{equation*}
Notice that the choice of the sign of $\xi_n$ makes $\xi$ a light-like future-pointing covector, pointing into $(0,T)\times\Omega$. Now assume $y_0\in U\subset (0,T)\times\partial\Omega$.
We choose
\begin{equation*}
    f(y)=-i\lambda\xi_ne^{i\lambda\phi(y,\xi')}\chi(y,\xi)
\end{equation*}
as the relative boundary value, where $\chi$ is a smooth cutoff function with support $U$, and $\chi=1$ on a subset $U_0\Subset U$.
In $(0,T)\times\Omega$ near $y_0$, the amplitudes $a_{00}$ and $a_{0j}$, 
$j=1,2,...,N$, solve the following transport equations:
\begin{equation*}
    \begin{cases}
        \mathcal{T}a_{00}=0, &\mathrm{in}\ (0,T)\times\Omega,\\
        a_{00}=\chi(y,\xi), & x^n=0,
    \end{cases}
\end{equation*}
and
\begin{equation*}
    \begin{cases}
        2i\mathcal{T}a_{0j}=-\square_ga_{0(j-1)}, &\mathrm{in}\ (0,T)\times\Omega,\\
        i\xi_na_{0j}=-\partial_na_{0(j-1)}, & x^n=0,
    \end{cases}
\end{equation*}
where the operator $\mathcal{T}$ is defined as
\begin{equation*}
    \mathcal{T}:=-\partial_0\phi\partial_0+g^{jk}\partial_j\phi\partial_k
    +\partial_n\phi\partial_n+\square_g\phi.
\end{equation*}
The amplitudes $a_{00}$ and $a_{0j}$, 
$j=1,2,...,N$ are supported in a neighborhood of the characteristics issued from
$y_0$ in the codirection $\xi(y_0)$. 
Note that $\partial_0\phi=-1=:\xi_0$, $\partial_\alpha\phi=\xi_\alpha$ and $\partial_n\phi=\xi_n$. 
% Let ${\xi^0}'$ be a 
% future-pointing timelike covector in $T_{x_0}^*\partial M$ at $x_0$.
As a result, in some
neighborhood of $y_0$, we have $u_\lambda^{(0)}$ solves 
$\square u^{(0)}_\lambda=O(\lambda^{-N})$, $a_\lambda|_{t=0}=0$ and 
$\partial_\nu u^{(0)}|_{(0,T)\times\partial\Omega}=f$.

We similarly construct the geometric optics solution to the following linearized 
system
\begin{equation*}
    (I_1)\begin{cases}
        \square_g u^{(1)} + q(u^{(0)})^m = 0, & \mathrm{in}\ (0,T)\times\Omega, \\
        \partial_\nu u^{(1)} = 0, & \mathrm{on}\ (0,T)\times\partial\Omega, \\
        u^{(1)}|_{t=0}=\partial_t u^{(1)}|_{t=0}=0, & (x',x^n)\in\Omega,
    \end{cases}
\end{equation*}
% where $u^{(0)}$ is the solution to 
% the unperturbed problem
% \begin{equation*}
%     (I_0)\begin{cases}
%         \square u^{(0)}= 0, & \mathrm{in}\ (0,T)\times\Omega \\
%         \partial_\nu u^{(0)} = f, & \mathrm{on}\ (0,T)\times\partial\Omega \\
%         u^{(0)}(0,x)=\partial_t u^{(0)}(0,x)=0. & x\in\Omega
%     \end{cases}
% \end{equation*}
%
% where we denote $y=(t,x)=(t,x_1,...,x_n)=(x_0,x_1,...,x_n)$, and 
% $u^{(0)}$ takes the form 
% $u^{(0)}(y)=e^{i\lambda \xi\cdot y}$, where 
% $\xi=(-1,\xi_1,...,\xi_{n})$, and
% construct a geometric optics approximation of $u^{(1)}$ 
locally of the form
\begin{equation*}
    u_\lambda^{(1)}(y):=e^{im\lambda\phi(y,\xi')}\sum_{j=0}^{N}\frac{1}{\lambda^j}a_j(y,\xi')+\mathcal{O}(\lambda^{-N-1}).
\end{equation*}
% where the boundary locally is given by $U:=\{x_n=0\}$. 
% near $x_0$ in $\Gamma$.
Then near $y_0$, according to ($I_1$), we compute that
% \begin{equation*}
%     \begin{split}
%     \square(e^{i\lambda\phi}a_\lambda)&=e^{i\lambda\phi}\square a_\lambda+i\lambda e^{i\lambda\phi}a_\lambda\square\phi-\lambda^2e^{i\lambda\phi}a_\lambda\left((\partial_t\phi)^2-\sum(\partial_{x_i}\phi)^2\right)\\
%         &\quad +2i\lambda e^{i\lambda\phi}\left(\partial_t\phi\partial_t a_\lambda-\sum\partial_{x_i}\phi\partial_{x_i}a_\lambda\right),        
%     \end{split}
% \end{equation*}
% and the boundary condition
% \begin{equation*}
%     \partial_\nu(e^{i\lambda\phi}a_\lambda)=e^{i\lambda\phi}\partial_\nu a_\lambda+i\lambda e^{i\lambda\phi}a_\lambda\partial_\nu\phi.
% \end{equation*}
\begin{equation*}
    \begin{cases}
        \sum_{j=0}^N\frac{1}{\lambda^j}\square_g a_j+2im\lambda\sum_{j=0}^{N}\frac{1}{\lambda^j}\mathcal{T} a_j+\mathcal{O}(\lambda^{-N-1})=-qa_\lambda^m, & \mathrm{in}\ (0,T)\times\Omega, \\
        im\lambda\xi_n\sum_{j=0}^{N}\frac{1}{\lambda^j}a_j=\sum_{j=0}^{N}\frac{1}{\lambda^j}\partial_{n}a_j, & x^n=0,
    \end{cases}
\end{equation*}
% \textbf{[Q22: the above equation can hardly hold exactly, probably only up to $\mathcal{O}(\lambda^{-N-1})$.]}
where the amplitudes $a_0$, and $a_l$, $l=1,2,...,$ solve the following 
transport equations:
\begin{equation}\label{Ta0}
    \begin{cases}
        \mathcal{T} a_0=0, & \mathrm{in}\ (0,T)\times\Omega, \\
        a_0=0, & x^n=0,
    \end{cases}
\end{equation}
and
\begin{equation}\label{Ta1}
    \begin{cases}
        \mathcal{T} a_l=\frac{1}{2im}(-qA_l-\square_g a_{l-1}), & \mathrm{in}\ (0,T)\times\Omega, \\
        im\xi_n a_l=\partial_{n}a_{l-1}, & x^n=0,
    \end{cases}
\end{equation}
% and
% \begin{equation}\label{Taj}
%     \begin{cases}
%         \mathcal{T} a_l=-\frac{1}{2im}\square_g a_{l-1}, & \mathrm{in}\ (0,T)\times\Omega \\
%         im\xi_n a_l=\partial_{n}a_{l-1}, & x^n=0
%     \end{cases}
% \end{equation} 
where $A_l$ is locally an $m$-degree homogeneous polynomial depending on $a_{00},...,a_{0(l-1)}$, eg., $A_1=a_{00}^m$.
The formal linearized NtD map in the boundary normal 
coordinates $(x',x^n)$ is given by
\begin{equation*}
    \left.\Lambda_q'f\right|_{U_0}=e^{im\lambda\phi(y,\xi')}\sum_{j=0}^{N}\frac{1}{\lambda^j}a_j(y,\xi)+\mathcal{O}(\lambda^{-N-1}).
\end{equation*}
% The expression for $\Lambda_{\tilde{q}}'f$ is similar, with $a_j$ replaced by 
% $\tilde{a}_j$. Then
% \begin{equation*}
%     (\Lambda_q'-\Lambda_{\tilde{q}}')f=e^{im\lambda\xi\cdot y}\sum_{j=0}^{N}\frac{1}{\lambda^j}(a_j-\tilde{a}_j)+\mathcal{O}(\lambda^{-N-1}).
% \end{equation*}
By the transport equation for $a_0$, we obtain $a_0=0$
% \begin{equation*}
%     \|a_0\|_{L^\infty(U)}\le C\left(\frac{1}{\lambda}+\varepsilon\|f\|_{C^{s+1}(U)}^m\right),
% \end{equation*}
% Thus, taking the limit $\lambda\to\infty$ and noticing that $f$ is bounded yields
% \begin{equation*}
%     \|a_0\|_{L^\infty(U)}\le C\varepsilon.
% \end{equation*}
and
\begin{equation*}
    \|a_1\|_{L^\infty(U)}\le C\left(\frac{1}{\lambda}+\lambda^{m(s+2)+1}\varepsilon\right).
\end{equation*}
Choose $\lambda=\varepsilon^{-\frac{1}{m(s+2)+2}}$ to minimize the right-hand side. Then
\begin{equation*}
    \|a_1\|_{L^\infty(U)}\le C\varepsilon^{\frac{1}{m(s+2)+2}}.
\end{equation*}
Repeating this type of argument with interpolation and Sobolev embedding theorem 
will establish the estimate for $a_j$ for all $j\ge0$:
$$
\|a_j\|_{L^\infty(U)}\le C\varepsilon^{\frac{1}{2^{j-1}}\frac{1}{m(s+2)+2}},\quad
\|a_j\|_{C^1(U)}\le C\varepsilon^{\frac{\mu}{2^{j-1}}\frac{1}{m(s+2)+2}},
$$
for any $0<\mu<1$.
Consider the transport equation of $a_0$ in (\ref{Ta0}). It follows from the 
boundary condition that $a_0=0$ and $\partial_0 a_0=\partial_\alpha a_0=\partial_n a_0=0$. 
Moreover, $g^{nj}=\delta^{nj}$ in the semigeodesic 
coordinates; thus, by the boundary condition of $a_1$ in 
(\ref{Ta1}), $\partial_0 a_1=\partial_\alpha a_1=0$ on 
$(0,T)\times\partial\Omega$. 
The transport 
equation for $a_1$ is given by
\begin{equation}\label{TEa1}
    2im\mathcal{T}a_1+\square_ga_0=-qA_1,
\end{equation}
which when restricted on $U$, becomes
\begin{equation}\label{qqq}
    2im(\partial_n\phi\partial_n a_1+\square_g\phi a_1)=-qA_1.
\end{equation}
The boundary condition of $a_2$ in (\ref{Ta1}) yields
\begin{equation*}
    \begin{split}
    \|\partial_na_1\|_{L^\infty(U)}&=C\|a_2\|_{L^\infty(U)}\le C\varepsilon^{\frac{1}{2}\frac{1}{m(s+2)+2}},\\
    \|\partial_\alpha a_2\|_{L^\infty(U)}&=C\|\partial_na_1\|_{C^1(U)}\le C\varepsilon^{\frac{\mu'}{2}\frac{1}{m(s+2)+2}},
\end{split}
\end{equation*}
for any $0<\mu'<1$, and we take $\mu'=1/2$. Therefore, (\ref{qqq}) implies
$$
\|q\|_{L^\infty(U_0)}\le C(\|\partial_na_1\|_{L^\infty(U)}+\|a_1\|_{L^\infty(U)})\le C\varepsilon^{\frac{1}{2}\frac{1}{m(s+2)+2}},
$$
% where we have assumed that $A_1$ is non-vanishing.
% \textbf{[Do you need $A_1$ to be non-vanishing?]} 
because $a_{00}=1$ on $U_0$. Similarly, the transport equation for $a_2$ becomes
\begin{equation*}
    2im(-\partial_0\phi\partial_0a_2+g^{\beta\alpha}\partial_\beta\phi\partial_\alpha a_2
    +\partial_n\phi\partial_n a_2+\square_g\phi a_2)=-\square_g a_1,
\end{equation*}
and the boundary condition of $a_3$ in (\ref{Ta1}) yields
\begin{equation*}
    \|\partial_n a_2\|_{L^\infty(U)}=C\|a_3\|_{L^\infty(U)}\le C\varepsilon^{\frac{1}{4}\frac{1}{m(s+2)+2}},
\end{equation*}
which implies
$$
\|\partial_n^2a_1\|_{L^\infty(U)}
\le C(\|a_1\|_{C^2(U)}+\|a_2\|_{L^\infty(U)}+\|a_2\|_{C^1(U)}+\|\partial_na_2\|_{L^\infty(U)})
% \le C(\|\partial_\alpha a_2\|_{L^\infty(U)}+\|\partial_n a_2\|_{L^\infty(U)}+\|a_2\|_{L^\infty(U)})
\le C\varepsilon^{\frac{1}{4}\frac{1}{m(s+2)+2}}.
$$
To obtain an expression of $\partial_n q$, we differentiate the transport equation
in (\ref{TEa1}) 
% \textbf{[Q23: can not differentiate (39).]} 
and evaluate it on $U$:
\begin{align*}
    2im(-&\partial_n\partial_0\phi\partial_0a_1-\partial_0\phi\partial_n\partial_0a_1\\
    &+\partial_ng^{jk}\partial_j\phi\partial_ka_1+g^{jk}\partial_n\partial_j\phi\partial_ka_1+g^{jk}\partial_j\phi\partial_n\partial_ka_1\\
    &+\partial_n^2\phi\partial_n a_1+\partial_n\phi\partial_n^2 a_1+\partial_n(\square_g\phi) a_1
+\square_g\phi \partial_na_1)+\partial_n\square_ga_0\\
=&-\partial_n qA_1-q\partial_nA_1, 
\end{align*}
which implies
\begin{align*}
& \|\partial_nq\|_{L^\infty(U_0)} \\
& \quad \le C(\|\partial_n^2a_1\|_{L^\infty(U)}+\|\partial_na_1\|_{L^\infty(U)}+\|\partial_na_1\|_{C^1(U)}+\|a_1\|_{L^\infty(U)}+\|q\|_{L^\infty(U_0)}) \\
& \quad  \le C\varepsilon^{\frac{1}{4}\frac{1}{m(s+2)+2}}.
\end{align*}
\end{proof}
\begin{Rmk}
    We should also emphasize that the uniformity of the estimate is guaranteed by the compactness of \([d,T-d]\times\partial\Omega\).
\end{Rmk}
%\begin{Rmk}
%    Boundary stability is established without using the Alessandrini-PIE type identity, as boundary information is directly obtained via local coordinates. In contrast to the interior case—where the identity enables localization via intersecting various beams near a target point—boundary neighborhoods are accessible directly, making the identity unnecessary.
%\end{Rmk}

\section{Design of Boundary Conditions}\label{se6}
Before proving the main results in Theorem \ref{Thm:main}, we discuss the design of the boundary functions \( f_j \), for \( j = 0, 1, 2, \dots, m \), in two separate cases: \( m = 2 \) and \( m \ge 3 \). This design ensures that the solution to the linear wave equation (\ref{11-4}) approximates the geometric optics solution (\ref{12-7}), and allows us to extract information from a neighborhood of a line (when \( m = 2 \)) or a point (when \( m \ge 3 \)) within the interior medium, a notion that will be made precise later.

\subsection{Case $m=2$}

Owing to the finite speed of propagation for the wave equation, 
the inverse problem is inherently limited in the regions of 
$\mathbb{R}^{1+n}$ where information about the coefficient $q$ can be 
recovered. To address this constraint, we now specify the geometric 
and analytical conditions under which the inverse problem is solvable.

% there are limitations on the areas of $\mathbb{R}^{1+n}$ where we can 
% obtain information in the inverse problem. Regarding this, we introduce 
% the requirements that the setting of the problem should meet.

%We denote
%\begin{equation*}
%    d:=2\ \mathrm{inf}\{r>0|\Omega\subset B_r(x),\ \mathrm{for\ some}\ 
%    x\in\mathbb{R}^n\},
%\end{equation*}
%where $B_r(x)$ is the ball of radius $r$ centered at $x\in\mathbb{R}^n$. 
%By Jung's theorem, $\mathrm{diam}(\Omega)\le d\le\mathrm{diam}(\Omega)
%\sqrt{2n/(n+1)}$. We then assume that
%\begin{equation*}
%    T\ge 2d+2\lambda
%\end{equation*}
%for some given $\lambda>0$. Then let
%\begin{equation*}
%    t_1=d+\lambda\quad\mathrm{and}\quad t_2=T-d-\lambda.
%\end{equation*}
%The parameter $\lambda$ will be used to establish small neighborhoods 
%of the times $t=0$ and $t=T$. Note that $[t_1,t_2]\subset(0,T)$. 
%We only require $q$ satisfy:
%\begin{Def}[Admissible coefficients]
%    Given $s\ge0$, say that the coefficient function 
%    $q\in C_c^\infty(\mathbb{R}\times\Omega)$ is admissible if
%    \begin{equation*}
%        \mathrm{supp}(q)\Subset [t_1,t_2]\times\Omega.
%    \end{equation*} 
%\end{Def}
%This requirement ensures that any point in the support of $q$ can be 
%reached by sending waves from $(0,T)\times\partial\Omega$, and that 
%the corresponding measurements can be detected on 
%$(0,T)\times\partial\Omega$ as well. We say these points are detectable.

Let $T=\mathrm{diam}(\Omega)+2d$,
where $d>0$ is a small constant.
% consistent with the definition of the domain $\mathbb{D}$ in equation (\ref{regionD}). 
% To simplify the subsequent analysis,
% we will apply a suitable coordinate 
% transform that enables the vectors involved in our construction to take 
% on a more explicit form in the new 
% coordinate system, without altering the essential features of the proof. 
% We let 
% $\mathrm{supp}(q)\subset\Omega\subset B_r(0)$, where 
% $r=\mathrm{diam}(\Omega)/2$ and $B_r(0)$ denotes the ball centered at $0$ 
% whose radius is $r$. 
As in the previous section, let $\xi=(1,\xi')\in\mathbb{R}^{1+n}$ be a 
non-zero light-like vector, and assume $p'=(t_0,x_0)\in(0,T)\times\partial\Omega$, 
with $t_0=T-d$. 
Denote by $\nu$ the outward unit normal vector to the boundary $\partial\Omega$ 
at the point $x_0$, and assume $\langle\xi',\nu\rangle<0$. 
We proceed to construct a boundary term $f\in C^{s+1}((0,T)\times\partial
\Omega)$ with $s>(n+1)/2$, such that the corresponding solution to the 
linear wave equation (\ref{11-4}) is close to the geometric optics 
approximation (\ref{12-7}). 

We define an extended domain $\Omega'\supset\Omega$, which is a slightly 
larger open set containing $\overline{\Omega}$. 
Then, there exists a point $(t_0',x_0')\in\gamma_{p',\xi}$, such that 
$t_0'\ne t_0$ and 
$(t_0',x_0')\in[d,T-d]\times\partial\Omega$.
As a consequence, for sufficiently small $\delta>0$, we have
\begin{equation}\label{ali21}
    H([0,t_0-t_0'],\delta)\subset(0,T)\times\Omega',
\end{equation}
where $H([0,t_0-t_0'],\delta)$ denotes a $\mathbf{hypertube}$ in 
$\mathbb{R}^{1+n}$, of 
width $2\delta$, centered along the segment of the light ray 
$\gamma_{p',\xi}$ over the time interval $[0,t_0-t_0']$; that is, 
% it is 
% the set of points located within distance $\delta$ of the curve traced by
the axis of the hypertube is defined by
\begin{equation*}
    \left\{(t,x)\ |\ (t,x)=\gamma_{p',\xi}(t),\ t\in[0,t_0-t_0']\right\}.
\end{equation*}

% It follows from (\ref{ex17}) and (\ref{ali21}) that there exists a constant
% $\rho>0$ such that 
% $$
% 0<t_0'-\rho<t_0+\rho<T,
% $$ 
% and
% \begin{equation}\label{ex22}
%     \mathrm{supp}(u_\tau(t,\cdot))\subset(0,T)\times\Omega',\quad\forall t\in(t_0'-\rho,t_0+\rho).
% \end{equation}
% Next, we choose a non-negative function $\zeta\in C^\infty(\mathbb{R}^{1+n})$ 
% such that
% \begin{equation}\label{ex23}
%     \zeta(t,x)=\begin{cases}
%         1,&\quad t_0'\le t\le t_0\\
%         0.&\quad t<t_0'-\rho\ \mathrm{or}\ t>t_0+\rho
%     \end{cases}
% \end{equation}
We are now ready to define the boundary function. To emphasize its 
dependence on the point
$p'$ and the light-like vector $\xi$, we write
\begin{equation}\label{ex24}
    f_{p',\xi}=\partial_\nu u_\tau^{(0)}(p',\xi),\quad\mathrm{on}\ (0,T)\times\partial\Omega,
\end{equation}
which is of the form given by (\ref{12-7}) and constructed in Subsection \ref{se4}. 
% which follows from (\ref{ex22}) that
% \begin{equation*}
%     \begin{split}
%     \mathrm{supp}(f_{p',\xi})\subset(t_0'-\rho,t_0+\rho)\times\Omega'\subset(0,T)\times\Omega',\\
%     f_{p',\xi}=\partial_\nu u_\tau^{(0)},\quad\mathrm{when}\ t\in[t_0',t_0].        
% \end{split}
% \end{equation*}

We now proceed to construct the geometric optics solutions and 
the corresponding boundary data for the wave equation, following the above 
framework. 
Fix $t_0=T-d$ and for each $x_0\in\partial\Omega$, define $p'=(t_0,x_0)$. 
Let $\xi_0=(1,\xi_0')\in\mathbb{R}^{1+n}$, where 
$\xi_0'\in S^{n-1}$ is a unit spatial direction. Then define the vectors
\begin{equation}\label{xim2}
    \xi_1=\xi_2=-\frac{1}{2}\xi_0.
\end{equation}
These are also light-like vectors in $\mathbb{R}^{1+n}$, and satisfy
\begin{equation*}
    \xi_0+\xi_1+\xi_2=0.
\end{equation*}
Set $y_0=p'$, and let $y_1=y_2=(t_0',x_0')$, so that the three rays $\gamma_{y_0,\xi_0}, \gamma_{y_1,\xi_1}, \gamma_{y_2,\xi_2}$ coincide.
% \begin{equation*}
%     \gamma_{y_0,\xi_0}=\gamma_{y_1,\xi_1}=\gamma_{y_2,\xi_2},\quad\mathrm{for}\ t\in[0,t_0-t_0'].
% \end{equation*}
As a result, we obtain the relation
\begin{equation*}
    u_{1\tau}^{(0)}u_{2\tau}^{(0)} \varphi_\tau = a_{0\tau}a_{1\tau}a_{2\tau}.
\end{equation*}
We define the associated boundary data by
\begin{equation}\label{ex24-1}
    f_j:=f_{y_j,\xi_j},\quad j=0,1,2.
\end{equation}
% Thus for $j=0,1,2$,
% \begin{equation*}
%     \mathrm{supp}(f_j)\subset(t_0'-\rho,t_0+\rho)\times\Omega'\subset(0,T)\times\Omega'.    
% \end{equation*}
% As (\ref{ex22}), we know that
% \begin{equation*}
%     \mathrm{supp}(u_{1\tau}^{(0)}u_{2\tau}^{(0)} \varphi_\tau)\subset
%     H((-\rho,t_0-t_0'+\rho),\delta)\subset(0,T)\times\Omega',
% \end{equation*}
% and we denote the hypertube as
% \begin{equation*}
%     H:=H((-\rho,t_0-t_0'+\rho),\delta)\cap (0,T)\times\Omega.
% \end{equation*}
It is important to emphasize that in this case, the norms of the light-like
vectors satisfy $1\le|\xi_j|\le2$ for $j=0,1,2$, 
which implies that their norms are uniformly bounded above and below by 
positive constants. This uniform bound allows us to absorb $|\xi_j|$ 
into the constant $\bar{C}$ appearing in subsequent estimates, a fact 
that plays a key role in the stability analysis.

\subsection{Case $m\ge3$}

We now proceed to select $\xi_j$ and $y_j$, for $j=0,1,2,...,m$, involved 
in the construction of the 
geometric optics solutions, following the approach in \cite{AliLauri}. 
We begin by presenting the case 
$m=3$; the general case for $m>3$ follows analogously later.

Let $p=(t_0,x_0)\in\mathbb{D}$, where $\mathbb{D}$ is given by (\ref{regionD}). 
% Our goal is to construct four geometric 
% optics solutions that serve as structured signals propagating from 
% the point $p$ back toward the boundary $(0,T)\times\partial\Omega$.
% Since $p\in\mathbb{D}$, 
Then, there exists a point 
$p^-=(t_1,x_1)\in[d,T-d]\times\partial\Omega$ 
% \textbf{[should it be $(d,T-d)\times\partial\Omega$?]} 
and a non-zero light-like vector $\xi^-\in\mathbb{R}^{1+n}$ such that 
$t_1<t_0$ 
and $p=\gamma_{p^-,\xi^-}(s_1)$ for some $s_1\in\mathbb{R}$, where $\gamma$
is the light-ray defined in (\ref{ex12}). Without loss of generality, 
we normalize the covector 
$\xi^-=(\xi_0^-,...,\xi_n^-)$ such that $\xi_0^-=-1$, so that $s_1=t_0-t_1$. 
Moreover, since $p\in\mathbb{D}$, it also admits a future-directed 
light-like path to the boundary. That is,
there exists a point $p^+:=(t_2,x_2)$ on $(0,T)\times\partial\Omega$, 
with $t_2>t_0$, and a light-like covector 
$\xi^+=(-1,\xi_1^+,...,\xi_n^+)$ such that $p=\gamma_{p^+,\xi^+}(-s_2)$,
where $s_2=t_2-t_0$.

% 20250702 12:30 修改至此

We proceed to select two additional light-like covectors as small 
perturbations of $\xi^-$, such that $\xi^+$ can be expressed as a linear 
combination of $\xi^-$ and these perturbations. To facilitate this 
construction, we perform a rotation of 
the spatial coordinates
$(x_1,...,x_n)\in\mathbb{R}^n$, such that in the rotated coordinate system, 
the covectors $\xi^+$ and $\xi^-$ are represented as 
\begin{equation}\label{ex29}
    \tilde{\xi}_0=(-1,\cos\theta,\sin\theta,0,...,0),\qquad 
    \tilde{\xi}_1=(-1,1,0,0,...,0),
\end{equation}
respectively, for some angle $\theta\in[0,2\pi]$. 

In general, it is possible that $\tilde{\xi}_0=\tilde{\xi}_1$, i.e., 
$\theta=0$, leading to degeneracy in the linear dependence structure. 
To avoid this, we perturb $\theta$ slightly by introducing a small angle 
$\theta_0>0$, and set $\theta=\theta_0$ or 
$\theta=2\pi-\theta_0$, ensuring $\tilde{\xi}_0\ne\tilde{\xi}_1$.
We now choose $\theta_0$ small enough so that the corresponding geodesic 
$\gamma_{p,\tilde{\xi}_0}$ still intersects the boundary of the domain in 
the time interval $\left[\frac{d}{2},T-\frac{d}{2}\right]$, that is,
\begin{equation*}
    \gamma_{p,\tilde{\xi}_0}\cap(0,T)\times\partial\Omega\subset\left[\frac{d}{2},T-\frac{d}{2}\right]\times\partial\Omega.
\end{equation*}
In fact, one can define
\begin{equation*}
    \psi:=\inf_{p\in\mathbb{D}}\sup\left\{\ \theta\ \left|\ \gamma_{p,\tilde{\xi}_0}\cap(0,T)\times\partial\Omega\subset\left[\frac{d}{2},T-\frac{d}{2}\right]\right.\times\partial\Omega\ \right\},
\end{equation*}
% \textbf{[For different $p$, the choice of coordinates is different. How the above definition makes sense?]} 
which quantifies the maximal allowable angular deviation that preserves 
boundary intersection. This quantity is independent of the specific 
choice of $p\in\mathbb{D}$, and
\begin{equation*}
    \psi=\mathcal{O}\left(\frac{d}{\mathrm{diam}(\Omega)}\right).
\end{equation*}
We emphasize that due to the fixed $d>0$ in the 
definition of $\mathbb{D}$ (see equation (\ref{regionD})), 
the maximal allowable value of $\theta_0$ 
satisfying the above condition admits a uniform positive lower bound over 
all $p\in\mathbb{D}$. We denote this bound by 
\begin{equation*}
    \theta_{\min}=\theta_{\min}(d)=\psi>0,
\end{equation*}
indicating its dependence solely on $d$. Hence, we may choose the angle 
$\theta$ from the interval
$\theta\in[\theta_{\min},2\pi-\theta_{\min}]$.
%% 这里待修改，要给theta=0施加一个小的扰动

We now define, for a small parameter 
$\sigma\in(0,1)$, the covectors
\begin{equation}\label{ex30}
    \tilde{\xi}_2=(-1,\sqrt{1-\sigma^2},\sigma,0,...,0),\qquad
    \tilde{\xi}_3=(-1,\sqrt{1-\sigma^2},-\sigma,0,...,0),
\end{equation}
which represent two small perturbations of the light-like vector 
$\tilde{\xi}_1$ in directions orthogonal to its spatial part.
According to \cite[Lem. 1]{Matti12}, the following linear relation holds:
\begin{equation}\label{ex31}
    \sigma^2\tilde{\xi}_0+\kappa_1\tilde{\xi}_1+
    \kappa_2\tilde{\xi}_2+\kappa_3\tilde{\xi}_3=0,
\end{equation}
where the coefficients $\kappa_1$, $\kappa_2$, $\kappa_3$ are given by
\begin{equation}\label{ex32}
    \begin{split}
    \kappa_1&=\frac{\sigma^2(\sqrt{1-\sigma^2}-\cos\theta)}{1-\sqrt{1-\sigma^2}},\\
    \kappa_2&=-\frac{\sigma\sin\theta+\sigma^2}{2}-\frac{1}{2}\frac{\sigma^2(\sqrt{1-\sigma^2}-\cos\theta)}{1-\sqrt{1-\sigma^2}},\\
    \kappa_3&=\frac{\sigma\sin\theta-\sigma^2}{2}-\frac{1}{2}\frac{\sigma^2(\sqrt{1-\sigma^2}-\cos\theta)}{1-\sqrt{1-\sigma^2}}. 
    \end{split}
\end{equation}
% with $b(\theta)=1-\cos\theta$. 
% 20250414 17:03至此
We will also denote $\kappa_0:=\sigma^2$ for convenience, 
% To return to the original coordinate system, 
and define the covector 
$\xi_0$ to be the representation of 
$\sigma^2\tilde{\xi}_0$, 
% after passing back to the original coordinate system, 
with $\xi_j$ for $j=1,2,3$ to be the analogous representations 
of the covectors $\kappa_j\tilde{\xi}_j$. Then $\xi_0=\sigma^2\xi^+$, 
$\xi_1=\kappa_1\xi^-$.

\begin{figure}[htbp]
    \centering
    \includegraphics[width=0.8\textwidth]{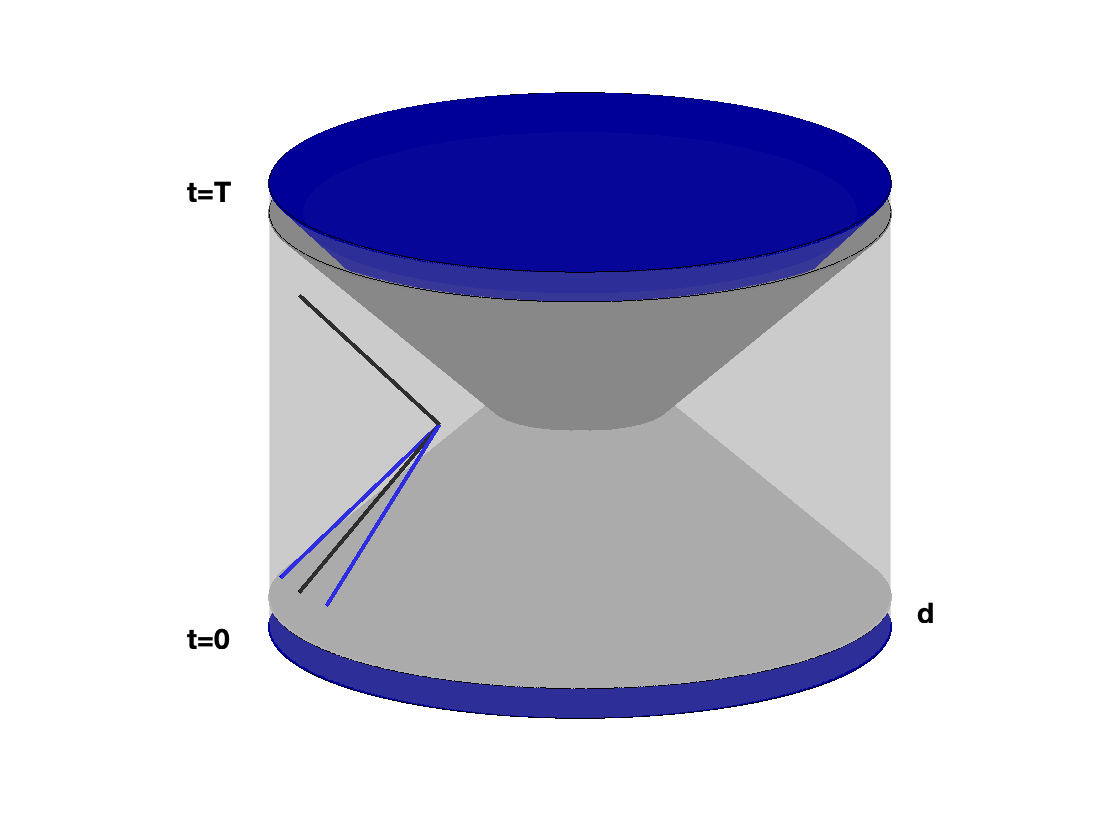}  % 插入图片，设置宽度为页面宽度的80%
    \caption{Illustration of covectors $\xi_j$ when $m=3$ in $(0,T)\times\mathbb{R}^2$.}  % 添加图片标题
    \label{fig:Xi}  
\end{figure}

%% 20250702 13:30 修改至此

We define
\begin{equation}\label{ex33}
    p_0=\gamma_{p,\xi_0}(s_2/\sigma^2),\quad
    p_1=\gamma_{p,\xi_1}(-s_1/\kappa_1).
\end{equation}
By construction, we have $p_0=p^+$ and $p_1=p^-$. For $j=2,3$,
we choose $\sigma>0$ sufficiently small so that the geodesics 
$\gamma_{p,\xi_j}$ intersect the boundary segment
$[\frac{d}{2},T-\frac{d}{2}]\times\partial\Omega$ at points 
$p_j$ that are close to $p_1$. The angle between $\gamma_{p,\xi_1}$ and 
$\gamma_{p,\xi_j}$ is given by $\arctan{\frac{\sigma}{\sqrt{1-\sigma^2}}}$, 
with a positive lower bound given by:
\begin{equation*}
    \psi':=\inf_{p\in\mathbb{D}}\sup\left\{\left.\arctan{\frac{\sigma}{\sqrt{1-\sigma^2}}}\ \right|\ \gamma_{p,\xi_j}\cap(0,T)\times\partial\Omega\subset\left[\frac{d}{2},T-\frac{d}{2}\right]\times\partial\Omega,j=2,3\right\},
\end{equation*}
which has the estimate given as
\begin{equation*}
    \psi'=\mathcal{O}\left(\frac{d}{\mathrm{diam}(\Omega)}\right),
\end{equation*}
independent of the choice of $p\in\mathbb{D}$.
% This implies that, due to the fixed constant $d>0$ in the 
% definition of $\mathbb{D}$ in (\ref{regionD}), we do not need to take the
% limit $\sigma\to0$. 
Therefore, there exists a uniform positive lower bound 
for the admissible range of $\sigma$, depending only on $d$, which we 
denote as
\begin{equation*}
    \sigma_{\min}=\sigma_{\min}(d)=\psi'>0.
\end{equation*}
Furthermore, from the expressions in (\ref{ex32}),
we observe that $|\xi_j|$ for $j=0,1,2,3$ also have uniform positive 
lower and upper bounds independent of $p$.
To ensure the asymptotic behavior of the coefficients $\kappa_j$,
it suffices to choose $\sigma_{\min}\ll\theta_{\min}$, in which case we 
can write:
\begin{equation*}
    \kappa_1=1-\cos\theta+\mathcal{O}(\sigma),\
    \kappa_2=-\frac{1}{2}(1-\cos\theta)+\mathcal{O}(\sigma),\
    \kappa_3=-\frac{1}{2}(1-\cos\theta)+\mathcal{O}(\sigma).
\end{equation*}
Hence, all $\kappa_j$, $j=0,1,2,3$, possess uniform positive 
lower and upper bounds that are independent of the point $p\in\mathbb{D}$.

We set $\sigma=\sigma_{\min}$, and then there exist $s_2,s_3\in\mathbb{R}$ such 
that $p_j=\gamma_{p,\xi_j}(s_j)$, $j=2,3,$ with 
$p_2,p_3\in[\frac{d}{2},T-\frac{d}{2}]\times\partial\Omega$. 
We define an extended domain $\Omega'\supset\Omega$, as in the case $m=2$, 
and we are now ready to choose the boundary data 
$f_{p_j,\xi_j}$ for $j=0,1,2,3$ according to (\ref{ex24}). 
% Note that in this case, 
% as in (\ref{ex22}), we have the support condition
% \begin{equation*}
%     \mathrm{supp}(u_{1\tau}^{(0)}u_{2\tau}^{(0)}u_{3\tau}^{(0)} \varphi_\tau)\subset
%     H(t_0,\delta)\subset\mathbb{D},
% \end{equation*}
% where $H(t_0,\delta)$ is a small hypercube neighborhood of the point $p$, 
% and we denote this hypercube by
% \begin{equation*}
%     H:=H(t_0,\delta).
% \end{equation*}
This completes the construction in the case $m=3$.

The case $m>3$ proceeds similarly. We define
\begin{equation*}
    \tilde{\xi}_1=\tilde{\xi}_4=\cdots=\tilde{\xi}_m=\frac{1}{m-2}(-1,1,0,...,0),
\end{equation*}
and set $p_1=p_4=p_5=\cdots=p_m=p^-$, $\kappa_1=\kappa_4=\kappa_5=\cdots=\kappa_m$. 
All other arguments follow as in 
the case $m=3$. 
Therefore, 
\begin{equation}\label{m3xij}
    \xi_j = \kappa_j\tilde{\xi}_j,\quad j=0,1,2,...,m.
\end{equation}
We emphasize that the norms 
$|\xi_j|$ are uniformly bounded from above and below by positive constants 
independent of the point $p\in\mathbb{D}$.
We define the associated boundary data of the form (\ref{12-7}) by
\begin{equation*}
    f_j:=f_{p_j,\xi_j},\quad j=0,1,2,...,m.
\end{equation*}
The rest of the construction is then similar to the case \(m=3\) above.

\section{Interior Stability}\label{se7}

The proof of the stability estimate will proceed in three steps, which can be started without discriminating the $m=2$ case and $m\geq 3$ case. It will diverge when necessary.
Throughout the argument, we continue to emphasize the uniformity 
of all involved constants.
% \textbf{[remark that you can start without discriminating the $m=2$ case and $m\geq 3$ case, and then diverge.]}
\subsection{Step 1}

We begin by estimating the right-hand side of the 
Alessandrini-PIE type identity in (\ref{12-6}) such that
\begin{equation*}
    \begin{split}
        \Bigg| \int_{(0,T)\times\partial\Omega} &\Lambda_q'\left( 
            \sum_{j\in S} f_j \right) \partial_\nu \varphi dS_xdt
        \Bigg| \\
        &\le \left\| \Lambda_q'\left( 
            \sum_{j\in S} f_j \right) 
            \right\|_{L^2((0,T)\times\partial\Omega)} \left\|
                \partial_\nu\varphi 
                \right\|_{L^2((0,T)\times\partial\Omega)} \\
            &\le \bar{C}(\Omega,T) \left\| \Lambda_q'\left( 
                \sum_{j\in S} f_j \right) 
                \right\|_{E^{s-\frac{1}{2}}((0,T)\times\partial\Omega)}
                \left\|
                \partial_\nu\varphi 
                \right\|_{L^\infty((0,T)\times\partial\Omega)} \\
            &\le \bar{C}(\Omega,T) \varepsilon\left\|    
                \sum_{j\in S} f_j  
                \right\|_{C^{s+1}((0,T)\times\partial\Omega)}^m
                \left\|
                \partial_\nu\varphi 
                \right\|_{L^\infty((0,T)\times\partial\Omega)} \\
            &\le \bar{C}(\Omega,T) \varepsilon |S|
                \sum_{j\in S}\left\| f_j  
                \right\|_{C^{s+1}((0,T)\times\partial\Omega)}^m
                \left\|
                \partial_\nu\varphi 
                \right\|_{L^\infty((0,T)\times\partial\Omega)},
    \end{split}
\end{equation*}
where 
$$
\varepsilon=\sup_{\|f\|_{C^{s+1}((0,T)\times\partial\Omega)
}=1}\|\Lambda_q'f\|_{E^{s-\frac{1}{2}}((0,T)\times\partial\Omega)}.
$$
We emphasize that the constant $\bar{C}(\Omega,T)$ is uniform since 
it depends only on $\Omega$ and $T$, and is independent on  
particular choice of $y_j$ or $\xi_j$. 

Let $\xi_j=(\xi_{j0},\xi_j')$, where $j=0,1,2,...,m,$ 
% \textbf{[Q24: What is $U$?]} 
and $\xi_j'\in\mathbb{R}^n$.
Then from equation (\ref{ex24-1}), we know that on 
$(0,T)\times\partial\Omega$,
\begin{equation*}
    f_j=\partial_\nu u_{j\tau}^{(0)}=i\xi_j' 
    e^{i\tau\xi_j\cdot y}\left(\sum_{k=0}^{N}
    \frac{a_{jk}(y)}{\tau^{k-1}}\right)\cdot\nu + 
    e^{i\tau\xi_j\cdot y}\left(\sum_{k=0}^{N}
    \frac{\partial_\nu a_{jk}(y)}{\tau^{k}}\right).
\end{equation*}
% In this context, we focus solely on the order relationships among the 
% parameters $\tau$, $\delta$, and the derivatives of $f_j$. 
Specifically, we observe that 
% \textbf{[Q25: partial derivatives, use $\partial$ not $d$.]}
\begin{equation}\label{12-12}
    {\frac{\partial^{s+1}}{\partial y^{s+1}}f_j}= \frac{\partial^{s+1}}{\partial y^{s+1}}(\partial_\nu u_{j\tau}^{(0)})
        =\bar{C}
        \sum_{l=0}^{s+2}%C_{s+2}^l e^{i\tau\xi_j\cdot y} 
        \sum_{k=0}^{N}\frac{\partial_\nu^l a_{jk}(y)}{\tau^{k-(s+2-l)}}
    .
\end{equation}
Then, combining estimates (\ref{ajkl}) and (\ref{12-12}), we obtain
\begin{equation*}
    \|{f_j}\|_{C^{s+1}((0,T)\times\partial\Omega)}= \|\partial_\nu 
    u_{j\tau}^{(0)}\|_{C^{s+1}((0,T)\times\partial\Omega)}
    \le\bar{C}|\xi_j|^{s+2} \sum_{l=0}^{s+2}\sum_{k=0}^{N}\frac{\tau^{s+2}}
    {(\tau\delta)^{k+l}}.
\end{equation*}
If we choose $\tau>1$ and $0<\delta<1$ such that 
$\tau\delta>1$, and recalling that $|\xi_j|$ for all $j=0,1,2,...,m$ 
have uniform positive lower and upper bounds, then the above inequality simplifies to
\begin{equation*}
    \|{f_j}\|_{C^{s+1}((0,T)\times\partial\Omega)}=\|\partial_\nu 
    u_{j\tau}^{(0)}\|_{C^{s+1}((0,T)\times\partial\Omega)}
    \le \bar{C}\tau^{s+2}.
\end{equation*}
From the choice of boundary values in (\ref{ex24-1}) and the 
estimate of geometric optics solutions in (\ref{12-8}), 
we therefore obtain
% \begin{equation*}
%         \left| \int_{(0,T)\times\partial\Omega} \Lambda_q'\left( 
%             \sum_{j\in S} f_j \right) \partial_\nu 
%             \varphi dS_xdt
%         \right| 
%             \le \bar{C}(\Omega,T)\varepsilon |S|^{2}\tau^{m(s+2)+1}.
% \end{equation*}
% We end up with 
the estimate for the right-hand side of (\ref{12-6}):
\begin{equation*}
    \left| -\frac{1}{2}\sum_{S\subset U} 
        (-1)^{|U\setminus S|} \int_{(0,T)\times\partial\Omega} 
        \Lambda_q' \left( \sum_{j\in S} f_j\right) 
        \partial_\nu \varphi dS_xdt \right|
    \le \bar{C}(\Omega,T,m)\varepsilon 
        \tau^{m(s+2)+1}.       
\end{equation*}
We highlight that the constant $\bar{C}$ here depends on $m$, and emphasize 
that the estimate grows at a rate of $O(2^m)$.

\subsection{Step 2}
In this step, we estimate the term 
% \begin{align*}
%     \left|\int_{(0,T)\times\Omega}q\prod_{j=0}^m a_{j\tau}dxdt
%         \right| 
%     &= \left|\int_{(0,T)\times\Omega}q 
%         \prod_{j=1}^mu_{j\tau}^{(0)} \varphi_\tau dxdt \right| \\
%     &= \left|\int_{(0,T)\times\Omega}q \prod_{j=1}^{m}
%         (u_{j\tau}^{(0)}-u_{j}^{(0)}+u_{j}^{(0)}) 
%         (\varphi_\tau-\varphi+\varphi) dxdt \right|,
% \end{align*}
\begin{equation*}
    \left|\int_{(0,T)\times\Omega}q\prod_{j=0}^m a_{j\tau}dxdt
        \right| 
    = \left|\int_{(0,T)\times\Omega}q 
        \prod_{j=1}^mu_{j\tau}^{(0)} \varphi_\tau dxdt \right|.
\end{equation*}
Because of the construction of the light-like covectors $\xi_j$,
$j=0,1,2,...,m$, it satisfies $\sum_{j=0}^m\xi_j=0$.
We denote $u_0^{(0)}:=\varphi$ and $u_{0\tau}^{(0)}:=\varphi_\tau$. 
Due to the uniform construction of the geometric optics solutions 
$u_{j\tau}^{(0)}$ and their corresponding boundary terms $f_j$, 
we have for all $j,k \in \{0,1,2,...m\}$ and $1\le p\le\infty$,
\begin{align*}
    \left\| u_{j\tau}^{(0)}-u_{j}^{(0)}\right\|_{L^p((0,T)\times\Omega)}
    &= \bar{C}(y_j,\xi_j) 
\left\| u_{k\tau}^{(0)}-u_{k}^{(0)}\right\|_{L^p((0,T)\times\Omega)},\\
    \left\| u_{j}^{(0)}\right\|_{L^p
((0,T)\times\Omega)}&= \bar{C}(y_j,\xi_j) \left\| u_{k}^{(0)}\right\|_{L^p
((0,T)\times\Omega)}.
\end{align*}
Applying the binomial expansion and using the triangle inequality, we derive
\begin{align*}
    \left|\int_{(0,T)\times\Omega}q\prod_{j=0}^m a_{j\tau}dxdt
        \right| 
    &\le \left|\int_{(0,T)\times\Omega}q 
        \prod_{j=1}^mu_{j}^{(0)}\varphi dxdt \right| \\
    &\quad + \bar{C}(y_j,\xi_j)\sum_{i=1}^{m}C_m^i\|q\|_{L^\infty
        } 
        \|u_{j\tau}^{(0)}-u_{j}^{(0)}\|_{L^p}^i\|u_{j}^{(0)}\|^{m-i}
        _{L^q},
\end{align*}
where $\frac{i}{p}+\frac{m-i}{q}=1$, and the norms are taken over
$(0,T)\times\Omega$.
% In the earlier chapter, we established that for well-posedness of the 
% nonlinear wave equation, the coefficient $q$ must satisfy a smallness 
% condition. 
Recall that we have assumed $\|q\|_{H^s(\Omega)}< C$, 
where $C>0$ is a constant.

We know that each $u_{j}^{(0)}$ satisfies (\ref{11-4}) with $f_j = \partial_\nu u_{j\tau}^{(0)}$, then,
\begin{equation}\label{12-10}
    \begin{cases}
        \square \left(u_j^{(0)}-u_{j\tau}^{(0)}\right)= -\square u_{j\tau}^{(0)}, 
            & \mathrm{in}\ (0,T)\times\Omega, \\
        \partial_\nu \left(u_j^{(0)}-u_{j\tau}^{(0)}\right) = 0, 
            & \mathrm{on}\ (0,T)\times\partial\Omega, \\
        \left(u_j^{(0)}-u_{j\tau}^{(0)}\right)(0,x)
            =\partial_t \left(u_j^{(0)}-u_{j\tau}^{(0)}\right)(0,x)=0,
            & x\in\Omega.
    \end{cases}
\end{equation}
%which is similar to the proof of the well-posedness of the NtD 
%map. By Theorem 3.1 in \cite{Zhai8}, $\left(u_j^{(0)}-u_{j\tau}^{(0)}\right)
%\in \bigcap_{0\le k\le s} W^{k,\infty}((0,T);H^{s-k}(\Omega))$, 
By Theorem 3.1 in \cite{Zhai8} and the geometric optics construction (cf. estimate %(cf. [\cite{UhlmannZhai},(35)]) 
(\ref{12-8})), for all $k\ge0$, we have
\begin{equation*}
    \begin{split}
        ||u_j^{(0)}||_{E^{k}((0,T)\times\Omega)}&\le \bar{C}||f_j||_{C^{k+1}((0,T)\times\partial\Omega)}\le \bar{C}\tau^{k+2},\\
        ||u_j^{(0)}-u_{j\tau}^{(0)}||_{E^{k}((0,T)\times\Omega)} &\le \bar{C}||\square u_{j\tau}^{(0)}||_{E^k((0,T)\times\Omega)}\le C(T,\Omega)\delta^{-N-2}\tau^{-N+k}.
    \end{split}
\end{equation*}
    % We emphasize that the constant $\bar{C}$ is uniform in $y_j$ and 
    % $\xi_j$, as guaranteed by the energy estimate and geometric optics 
    % construction.

By the Sobolev embedding theorem, we have
\begin{equation*}
    ||u||_{H^{k}((0,T)\times\Omega)}\le \bar{C}||u||_{E^{k}((0,T)\times\Omega)}.
\end{equation*}
Then we let $p=q=m$ and use Sobolev embedding $H^{k'}((0,T)\times\Omega)\to 
L^m((0,T)\times\Omega)$, provided $2k'>1+n$ and $m\ge2$. 
By further defining
\begin{equation*}
    k'(n)=\begin{cases}
        \frac{n+2}{2}, & n=2,4,6,...,\\
        \frac{n+3}{2}, & n=3,5,7,...,
    \end{cases}
\end{equation*}
which satisfies $k'(n)\le s$, coinciding with the assumption that $s\in\mathbb{N}$ and $s>(n+1)/2$, so that the embedding holds for any $n\ge2$, we obtain
% \begin{equation*}
%     H^{k'(n)}((0,T)\times\Omega)\to L^m((0,T)\times\Omega).
% \end{equation*}
% Hence, we obtain
\begin{equation*}
    \begin{split}
        ||u_j^{(0)}||_{L^m((0,T)\times\Omega)}&\le\bar{C}||u_j^{(0)}||_{H^{k'}((0,T)\times\Omega)}\le \bar{C}\tau^{k'+2},\\
        ||u_j^{(0)}-u_{j\tau}^{(0)}||_{L^m((0,T)\times\Omega)}&\le\bar{C}||u_j^{(0)}-u_{j\tau}^{(0)}||_{H^{k'}((0,T)\times\Omega)} \le \bar{C}\delta^{-N-2}\tau^{-N+k'}.
    \end{split}
\end{equation*}

We now return to the main estimate: %Recall from the previous passage:
\begin{align*}
    \left|\int_{(0,T)\times\Omega}q\prod_{j=0}^{m}a_{j\tau}dxdt
        \right| 
    &\le \left|\int_{(0,T)\times\Omega}q 
        \prod_{j=1}^{m}u_{j}^{(0)}\varphi dxdt \right| \\
    &\quad + \bar{C}\tau^{\frac{1}{2}m(n+7)} \left[\left(
        1+\delta^{-N-2}\tau^{-N-2}\right)^{m}-1\right],
\end{align*}
where we used the upper bound $k'(n)+2\le\frac{1}{2}(n+7)$.
When $N\ge m(s+2)$, 
\begin{equation}\label{step2}
    \left|\int_{(0,T)\times\Omega}q\prod_{j=0}^{m}a_{j\tau}dxdt
    \right| \le\bar{C}\left(\varepsilon\tau^{m(s+2)+1}+\tau^{\frac{1}{2}m(n+7)-N-2}\delta^{-N-2}\right).
\end{equation}
This concludes the estimate in Step 2.

\subsection{Step 3}

Recalling the choice of boundary values, 
in this step, we estimate the approximation behavior 
of the integral
\begin{equation*}
    \left| \int_{(0,T)\times\Omega}q(y)\prod_{j=0}^{m}a_{j\tau}(y-y_0)dy\right|. 
\end{equation*}
By (\ref{aj0}), for $j=0,1,...,m$, the first term of the amplitude function $a_{j\tau}$ is given as
$$
        a_{j0}(y) = \chi_\delta\left(\xi_j\cdot y\right)
            \prod_{l=1}^{n-1}\chi_\delta\left(w_{jl}\cdot y\right).
$$
% and their product is
% $$
%         \prod_{j=0}^{m} a_{j0}(y) = \prod_{j=0}^{m}\left(\chi_\delta\left(\xi_j\cdot y\right)
%         \prod_{l=1}^{n-1}\chi_\delta\left(w_{jl}\cdot y\right)\right).
% $$
We extend $q$ from $[d,T-d]\times\Omega$ to $[d,T-d]\times\Omega'$ such 
that $q$ is $C_c^1$ in $[d,T-d]\times\Omega'$ with 
$$
\|q\|_{C^1([d,T-d]\times(\Omega'\setminus\Omega))}\le C\|q\|_{C^1([d,T-d]\times\partial\Omega)}.
$$
Below, we choose solutions for different $m$ such that for $m=2$, the intersection of the support of solutions is a neighbourhood of a line segment, while for $m\ge3$, the intersection is a neighbourhood of a point. This lead to varied estimates. The $m\ge3$ case will be specified first in this part.
% \textbf{[Do you need $q$ is compactly supported in $\Omega'$?]}
% \textbf{[Recall your different choices of solutions for $m=2$ and $m\geq 3$.]}
\subsubsection{Case $m\ge3$}

Similar to Lemma 3 in \cite{LassasP8}, we have the following approximation 
property.
The proof of the lemma follows the line of \cite[Lemma 3]{LassasP8} and can be referred in Appendix A.
\begin{Lem}\label{Lem:4}
    Let $b\in C^1((0,T)\times\Omega')$ where $\Omega'\supset\Omega$ 
    is a bounded open subset of $\mathbb{R}^n$, $m\ge2$, $m\in\mathbb{N}$ and 
    $\delta>0$. The following estimate
    \begin{equation*}
        \left| b(y_0) -C(n,y_0)\delta^{-(n+1)}\int_{(0,T)\times\Omega'}
        b(y)\prod_{j=0}^{m}a_{j0}(y-y_0)dy\right|\le
        C'(n)\|b\|_{C^1}\delta
    \end{equation*}
    holds true for all $y_0=(t_0,x_0)\in(0,T)\times\overline{\Omega}$, where 
    $\xi_j$ is selected as in (\ref{m3xij}), and $w_{jl}$ 
    is defined as in (\ref{ex15}) with $|w_{jl}|=1$, and therefore
    $C(n,y_0)$, $C'(n)$ are constants with uniform positive lower 
    and upper bounds independent on $y_0$. In particular, the integral 
    on the left-hand side converges uniformly to $b$ when $\delta\to0^+$.
\end{Lem}
%\fix{Change " Chapter 6" to concrete equation. }

Since
$$
        a_{j\tau}(y)= a_{j0}(y)+\sum_{k=1}^{N}
                  a_{jk}(y)\tau^{-k}+\mathcal{O}((\delta\tau)^{-(N+1)}),
$$
% \textbf{[Q26: $\mathcal{O}(\tau^{-N-1})$? This happens a lot throughout the paper. Also should $\delta$ be involved?]} 
we denote the remainder by $\mathcal{R} a_j(y):=a_{j\tau}(y)-a_{j0}(y)$, and observe that
$$
    \mathcal{R} a_j(y)\sim(\delta\tau)^{-1}+(\delta\tau)^{-N}+
(\delta\tau)^{-(N+1)}.
$$
Choosing $\delta<1$ such that $\delta\tau>1$, we simplify:
$\mathcal{R} a_j(y)\sim (\delta\tau)^{-1}$. Then, expanding the 
product gives
\begin{align*}
&\quad\left| \int_{(0,T)\times\Omega'}\ q(y)\left(\prod_{j=0}^{m}a_{j\tau}(y-y_0)
      -\prod_{j=0}^{m}a_{j0}(y-y_0) \right)dy\right| \\
      &\le C(m,T,\Omega) \delta^{n+1} ((1+(\delta\tau)^{-1})^{m+1}-1)\\
      &\le C(m,T,\Omega) \delta^{n+1} (\delta\tau)^{-1}.
\end{align*}
% \begin{align*}
%     &\quad\left| \int_{(0,T)\times\Omega'}qa_{0\tau}a_{1\tau}\cdots a_{m\tau}dxdt - \int_{
%         (0,T)\times\Omega'}qa_{00}a_{10}\cdots a_{m0}dxdt\right| \\
%     &= \left|\int_{(0,T)\times\Omega'}q(a_{00}+\mathcal{R} a_0)\cdots 
%         (a_{m0}+\mathcal{R} a_m)dxdt - \int_{(0,T)\times\Omega'}qa_{00}a_{10}
%         \cdots a_{m0}dxdt\right| \\
%     &\sim \left|\int_{(0,T)\times\Omega'}q(a_{j0}+\mathcal{R} a_j)^{m+1}dxdt 
%         - \int_{(0,T)\times\Omega'}qa_{j0}^{m+1}dxdt\right| \\
%     &= \left|\int_{(0,T)\times\Omega'}q\left[\binom{m+1}{m}a_{j0}^m\mathcal{R} a_j 
%         + \binom{m+1}{m-1}a_{j0}^{m-1}\mathcal{R} a_j^2+\cdots
%         + \binom{m+1}{1}a_{j0}\mathcal{R} a_j^m 
%         + \mathcal{R} a_j^{m+1}\right]dxdt\right| \\
%     &\le C(T,\Omega)\|q\|_{L^\infty((0,T)\times\Omega')} \left[
%         \binom{m+1}{1}(\delta\tau)^{-1} + 
%         \binom{m+1}{2}(\delta\tau)^{-2}\right.\\
%     &\qquad\qquad\qquad\qquad\qquad \left.+\cdots
%         +\binom{m+1}{m}(\delta\tau)^{-m} + 
%         \binom{m+1}{m+1}(\delta\tau)^{-(m+1)}
%         \right]\\
%     &\le C(T,\Omega)\varepsilon_0 ((1+(\delta\tau)^{-1})^{m+1}-1)\\
%     &\le C(m,T,\Omega)\varepsilon_0 (\delta\tau)^{-1}.
% \end{align*}
Combining with the pointwise approximation (\ref{step2}) and boundary stability estimates 
in Theorem \ref{Thm:2}, 
we conclude
\begin{align*}
    |q(y_0)|&\le \left| C(n,y_0)\delta^{-(n+1)}
        \int_{[d,T-d]\times\Omega}q(y)\prod_{j=0}^{m} a_{j\tau}(y-y_0)dy\right|\\
        &\quad + \left| q(y_0)-C(n,y_0)\delta^{-(n+1)}
        \int_{[d,T-d]\times\Omega'}q(y)\prod_{j=0}^{m} a_{j0}(y-y_0)dy\right|\\
        &\quad + C(n,y_0)\delta^{-(n+1)}\left| 
            \int_{[d,T-d]\times\Omega'}q(y)\left(\prod_{j=0}^{m} a_{j\tau}(y-y_0) - \prod_{j=0}^{m} a_{j0}(y-y_0)\right)dy\right| \\
        &\quad + C(n,y_0)\delta^{-(n+1)}\left| 
            \int_{[d,T-d]\times(\Omega'\setminus\Omega)}q(y)\prod_{j=0}^{m} a_{j\tau}(y-y_0)dy\right| \\
    % &\le C(m,n,\Omega,T)\delta^{-(n+1)}
    %         (\delta^{-N-2}\tau^{-N-2+\frac{1}{2}m(n+7)}
    %         +\varepsilon\tau^{m(s+2)+1})\\
    %     &\quad + C(n)\delta
    %         + C(m,n,\Omega,T)%\delta^{-(n+1)}
    %         (\tau\delta)^{-1}+C(m,n,\Omega,T)\varepsilon^\frac{1}{8}\\
    &\le C(m,n,\Omega,T)\left[\delta^{-(n+1)}
        \left(\delta^{-N-2}\tau^{-N-2+\frac{1}{2}m(n+7)}
            +\varepsilon\tau^{m(s+2)+1}
            \right)+\delta+(\tau\delta)^{-1}+\varepsilon^{\frac{1}{4}\frac{1}{m(s+2)+2}}
        \right].
\end{align*}
% 以下最后一步还需要考虑！！！！以及还有翟老师说的一致性的Matti证明
We now optimize the parameters.
% In previous chapters, we take the upper bound of the coefficient $q$ as a 
% constant. 
Let $p=m(s+2)+1$, $\delta=\varepsilon^\mu$ and $\tau=\varepsilon^{-2\mu}$ with $\mu>0$, then $1-\mu(n+1)-2\mu p>0$. Let $\mu=\frac{1}{n+4p+4}<{\frac{1}{4}\frac{1}{m(s+2)+2}}$ and 
$N>m(n+7)+n$.
% \begin{equation*}
%     |q(y_0)|\le C(m,n,\Omega,T)\left(
%         % \varepsilon^{-\frac{\sigma}{2}\left(m(n+7)+2n+2\right)}+
%         \varepsilon^{-\sigma\left(m(s+2)+n+2\right)+1}
%         +\varepsilon^{-\sigma(n+1)}
%         +\varepsilon^\frac{1}{8}
%         +\varepsilon^{\sigma}\right).
% \end{equation*}
% To ensure that the leading order term dominates, we require $-\sigma\left(m(s+2)+n+2\right)+1=\sigma$, 
% and then $\sigma=\frac{1}{m(s+2)+n+3}<\frac{1}{8}$. 
Substituting this into the estimate gives
\begin{equation*}
    \|q\|_{L^\infty(\mathbb{D})}\le\mathcal{O}(
        \varepsilon^\frac{1}{4m(s+2)+n+8}).
\end{equation*}

%% 20250702 18:46 修改至此

\subsubsection{Case $m=2$}
We estimate the approximation behavior of the integral 
$$
\left| \int_{(0,T)\times\Omega}q(y)\prod_{j=0}^{2}a_{j\tau}(y-y_0)dy\right|.
$$ 
To this end, we introduce the light-ray transform as follows.

Let $\Omega\subset\mathbb{R}^n$ and
$\xi=(\xi_0,\xi')\in\mathbb{R}^{1+n}$ which is light-like
satisfying the condition $|\xi_0|=|\xi'|$.
% i.e., $\xi$ lies on the light cone. 
For simplicity, we define %the light cone on the unit sphere as 
$S_1^n:=\{\xi=(\xi_0,\xi')\ |\ \xi'\in S^{n-1}, |\xi_0|=1\}$.  
% \textbf{[Q27: why call it the light cone on the unit sphere?]}

A line starting at a point $y_0=(t_0,x_0)\in(0,T)\times\partial\Omega$ with 
direction $\xi^\sharp$ is denoted by $t\mapsto\gamma(t,y_0,\xi^\sharp):=y_0+t\xi^\sharp$, $t\in\mathbb{R}$. 
Let $\tau(y_0,\xi^\sharp)$ denote the exit time of the line $\gamma$ from the 
domain $(0,T)\times\Omega$, i.e., the smallest positive time such that 
$\gamma(t,y_0,\xi^\sharp)\notin(0,T)\times\Omega$ for all $t\ge\tau(y_0,\xi^\sharp)$. 
We assume that 
the domain $\Omega$ is strictly convex, so that each line exits the 
domain in finite time and does so transversely. %We denote $t_0':=t_0+\tau(y,\xi)$.

For a function $G=G(y)=G(t,x)$, 
the light-ray transform
$\mathcal{L}(G)$ is defined as:
\begin{equation*}
    \mathcal{L}(G)(y_0,\xi^\sharp)=\int_{t_0}^{t_0+\tau(y_0,\xi^\sharp)}G(y_0+s\xi^\sharp)ds,
\end{equation*}
where we have assumed $y_0\in(0,T)\times\partial\Omega$, according to \cite{lightray}. 
% \textbf{[Q29: Should the integral be on $\int_0^{\tau(y,\xi)}$?]}
We establish the following lemma. The detailed proof can 
be found in Appendix A.
% In this work, we restrict our attention to functions with compact 
% support. Therefore, instead of integrating over all of $\mathbb{R}$, 
% it suffices to integrate over a finite interval. For instance, the 
% integration in the light-ray transform can be restricted to an interval 
% $[t_0,t_0']\subset\mathbb{R}$.

\begin{Lem}\label{Lem:5}
    Let $G\in C^1((0,T)\times\Omega)$, $y_0=(t_0,x_0)\in
    (0,T)\times\partial\Omega$ and $\delta>0$. There exist constants 
    $C(n,y_0)>0$ and $C'(n)>0$ such that the following estimate holds:
    \begin{equation*}
        % \left| \mathcal{L}(G)(y_0,\xi^\sharp) -C(n,y_0)\delta^{-n}\int_{(0,T)\times\Omega}
        % G(y)\prod_{j=0}^{2}\left(\chi_\delta\left(\xi_j\cdot(y-y_0)\right)
        % \prod_{l=1}^{n-1}\chi_\delta\left(w_{jl}\cdot(y-y_0)\right)\right)dy\right|\le
        % C'(n)\|G\|_{C^1}\delta.
        \left| \mathcal{L}(G)(y_0,\xi^\sharp) -C(n,y_0)\delta^{-n}\int_{(0,T)\times\Omega}
        G(y)\prod_{j=0}^{2}a_{j0}(y-y_0)dy\right|\le
        C'(n)\|G\|_{C^1}\delta.
    \end{equation*}
    In particular, the integral 
    on the left-hand side converges uniformly to $\mathcal{L}(G)$ as
    $\delta\to0^+$.
\end{Lem}
% \textbf{[Q30: $\xi$ v.s. $\xi^\sharp$. Notations need to be consistent.]}

For case $m=2$, we assume that $q(x)$ is independent of $t$. 
% The reason behind this assumption will be clarified in the discussion below. 
We begin by introducing the X-ray transform and several related concepts; 
see \cite{MaSalo} for more details. 
We denote 
the Euclidean inner product
%  on tangent or cotangent vectors 
by $\langle \cdot,\cdot\rangle$, and define the norm by $|\cdot|:=
\langle\cdot,\cdot\rangle^{1/2}$. 
The unit sphere bundle of $\Omega$, denoted $S\Omega$, is defined by
$S\Omega:=\bigcup_{x\in\Omega}S_x\Omega$, where $
    % S_x\Omega:=\{(x,v)\in T_x\Omega;|v|=1\}.
    S_x\Omega:=\{x\in\Omega,v\in\mathbb{R}^{n},|v|=1\}$.
% \textbf{[Q31: $v\in \mathbb{R}^n$.]}
% This is a smooth manifold of dimension $2n-1$. 
The boundary of the unit 
sphere bundle, denoted $\partial(S\Omega)$, is given by 
$
\partial(S\Omega)=\{(x,v)\in S\Omega;x\in\partial\Omega\}$,
and is 
the union of the sets of inward and outward pointing vectors:
$
    \partial_\pm S\Omega=\{(x,v)\in\partial(S\Omega);
    \pm\langle v,\nu\rangle\le0\}$,
where $\nu$ denotes the outward unit normal vector to $\partial\Omega$. 
A unit speed line starting at a point $x\in\Omega$ and traveling in 
the direction $v\in S^n$ is denoted by $t\mapsto\gamma(t,x,v):=x+tv$. 
Let $\tau(x,v)$ denote the exit time of the line $\gamma$ from the 
domain $\Omega$, i.e., the smallest positive time such that 
$\gamma(t,x,v)\notin\Omega$ for all $t\ge\tau(x,v)$. 

The X-ray\ transform $\mathcal{X}(F)$ of a function 
$F:\Omega\to\mathbb{R}$ is then defined by integrating $F$ along the 
line paths:
\begin{equation*}
    \mathcal{X}(F)(x,v)=\int_0^{\tau(x,v)}F(\gamma(t,x,v))dt,
\end{equation*}
where $(x,v)\in\partial_+S\Omega$.
Therefore, the light-ray transform 
$\mathcal{L}(q)$ naturally degenerates into the X-ray transform 
$\mathcal{X}(q)$, i.e., $\mathcal{X}(q)(x_0,\xi')=\mathcal{L}(q)(y_0,\xi^\sharp)$. By Lemma \ref{Lem:5}, we know that
$$
\left| \mathcal{X}(q)(x_0,\xi')-C(n,x_0)\delta^{-n} 
        \int_{(0,T)\times\Omega}q(y)\prod_{j=0}^{2}a_{j\tau}(y-y_0)dy\right|\le C(n)\|q\|_{C^1}\delta.
$$
% From equations (\ref{12-7}) and (\ref{12-9}), we know that
% \begin{equation*}
%     \begin{split}
%         a_{j0}(y) &= \chi_\delta\left(\xi_j\cdot y\right)
%             \prod_{l=1}^{n-1}\chi_\delta\left(w_{jl}\cdot y\right), \quad j=0,1,2,\\
%         a_{00}(y)a_{10}(y)a_{20}(y) &= \prod_{j=0}^{2}\left(\chi_\delta\left(\xi_j\cdot y\right)
%         \prod_{l=1}^{n-1}\chi_\delta\left(w_{jl}\cdot y\right)\right).
%     \end{split}
% \end{equation*}
Similarly to the case $m\ge3$, 
% we have the following estimate:
% \begin{align*}
%     &\quad\left| \int_{(0,T)\times\Omega}qa_{0\tau}a_{1\tau}a_{2\tau}dxdt - \int_{
%         (0,T)\times\Omega}qa_{00}a_{10}a_{20}dxdt\right| \\
%     &= \left|\int_{(0,T)\times\Omega}q(a_{00}+\mathcal{R} a_0)(a_{10}+\mathcal{R} a_1)
%         (a_{20}+\mathcal{R} a_2)dxdt - \int_{(0,T)\times\Omega}qa_{00}a_{10}
%         a_{20}dxdt\right| \\
%     &\sim \left|\int_{(0,T)\times\Omega}q(a_{j0}+\mathcal{R} a_j)^{3}dxdt 
%         - \int_{(0,T)\times\Omega}qa_{j0}^{3}dxdt\right| \\
%     &= \left|\int_{(0,T)\times\Omega}q\left[3a_{j0}^2\mathcal{R} a_j 
%         + 3a_{j0}(\mathcal{R} a_j)^2
%         + (\mathcal{R} a_j)^{3}\right]dxdt\right| \\
%     &\le C(T,\Omega)\|q\|_{L^\infty(\Omega)} \left[
%         3(\delta\tau)^{-1} + 
%         3(\delta\tau)^{-2}+ 
%         (\delta\tau)^{-3}
%         \right]\\
%     &\le C(T,\Omega) (\delta\tau)^{-1}.
% \end{align*}
% \begin{equation*}
%     \left| \int_{(0,T)\times\Omega}q(y)\left(\prod_{j=0}^{2}
% a_{j\tau}(y-y_0) -\prod_{j=0}^{2}a_{j0}(y-y_0)\right)dy\right|\le C(T,\Omega) \delta^n(\delta\tau)^{-1}.
% \end{equation*}
% Thus, 
we have the following upper bound for the X-ray transform:
\begin{align*}
    |\mathcal{X}(q)(x_0,\xi')|
    &\le \left| C(n,x_0)\delta^{-n}
        \int_{(0,T)\times\Omega}q(y)\prod_{j=0}^{2}a_{j\tau}(y-y_0)dy\right|\\
        &\quad + \left| \mathcal{X}(q)(x_0,\xi')-C(n,x_0)\delta^{-n}
        \int_{(0,T)\times\Omega}q(y)\prod_{j=0}^{2}a_{j0}(y-y_0)dy\right|\\
        &\quad + C(n,x_0)\delta^{-n}\left| 
            \int_{(0,T)\times\Omega}q(y)\left(\prod_{j=0}^{2}a_{j\tau}(y-y_0) - \prod_{j=0}^{2}a_{j0}(y-y_0)\right)dy\right| \\
    % &\le C(n,\Omega,T)\delta^{-n}
    %         (\delta^{-N-2}\tau^{-N-2+n+7}
    %         +\varepsilon\tau^{2(s+2)+1})\\
    %     &\quad + C(n)\delta
    %         + C(n,\Omega,T)
    %         (\tau\delta)^{-1}\\
    &\le C(n,\Omega,T)\left[\delta^{-n}
        \left(\delta^{-N-2}\tau^{-N-2+n+7}
            +\varepsilon\tau^{2(s+2)+1}
            \right)+\delta+(\tau\delta)^{-1}
        \right].
\end{align*}
 % \textbf{[Q33: mention where you have used Lemma 5.]}
% 以下最后一步还需要考虑！！！！以及还有翟老师说的一致性的Matti证明
% In past chapters we take the upper bound of the coefficient $q$ as a 
% constant. 
Let $\delta=\varepsilon^a$ and $\tau=\varepsilon^{-b}$
where $a,b>0$.
% This gives us the estimate:
% \begin{equation*}
%     |\mathcal{X}(q)(x_0,\xi')|\le C(n,\Omega,T)\left[\varepsilon^{-\sigma n}
%         \left(\varepsilon^{1-\sigma(2(s+2)+1)}
%         +1\right)
%         +\varepsilon^{\sigma}\right].
% \end{equation*}
%Here we only consider $q=q(x)$ which is 
%independent of the temporal variable, so we have actually obtained 
%the estimate of $|\mathcal{X}(q)(x_0,\xi')|$. 
Thus, we take 
$1 - a n - b(2s+5) = b-a = a$ and $N>3n+13$. We have proved 
the estimate for every $x_0\in\partial\Omega$ and $\xi'\in S^{n-1}$, so
from here, we derive the stability estimate:
\begin{equation*}
    \|\mathcal{X}(q)\|_{L^\infty(\partial_+S\Omega)}\le\mathcal{O}(
        \varepsilon^\frac{1}{4s+n+11}),
\end{equation*}
where $\partial_+S\Omega$ is defined as before. Therefore:
\begin{equation*}
    \|\mathcal{X}(q)\|_{L^2(\partial_+S\Omega)}\le\mathcal{O}(
        \varepsilon^\frac{1}{4s+n+11}).
\end{equation*}
Given that $q\in H^s(\Omega)$ with $s>\frac{n+1}{2}$, there exists a constant 
$M>0$ such that
\begin{equation*}
    \|\mathcal{X}(q)\|_{H^2(\partial_+S\Omega)}\le M,
\end{equation*} 
according to Lemma 2.3 in \cite{MaSalo}. Using interpolation between Sobolev spaces, we obtain
\begin{equation*}
    \|\mathcal{X}(q)\|_{H^1(\partial_+S\Omega)}^2\le C
\|\mathcal{X}(q)\|_{L^2(\partial_+S\Omega)}
\|\mathcal{X}(q)\|_{H^2(\partial_+S\Omega)},
\end{equation*}
which, by Proposition 3.1 in 
\cite{Taylor11} and Theorem 4.7.8 in 
\cite{Paternain}, leads to the estimate:
\begin{equation*}
    \|q\|_{L^2(\Omega)}\le\mathcal{O}(
        \varepsilon^\frac{1}{2(4s+n+11)}).
\end{equation*}
Finally, by $\|q\|_{H^s(\Omega)}<C$ and interpolation inequality, we derive
\begin{equation*}
    \|q\|_{L^\infty(\Omega)}\le\mathcal{O}(
        \varepsilon^{\frac{1}{2(4s+n+11)}\cdot\frac{2s-n}{2s}}). 
\end{equation*}

\section{Numerical Experiments}\label{se8}
Based on the above theoretical results, in this section, we propose a neural network inversion algorithm by the first-order linearized NtD map, which yields a least squares method handling both time-independent and time-dependent coefficients as an illustration. For simplicity's sake, we consider the semilinear wave equation \eqref{2-1} by fixing the nonlinearity exponent \( m = 3 \) and the spatial dimension \( n = 2 \) for all experiments. We note that in \cite{harju2025} the authors employ higher-order linearization to recover the unknown time-independent coefficient with the same nonlinearity and spatial setting.

\subsection{Neural Network Architecture for the Unknown Coefficient }\label{se8.2}
In the proposed algorithm, we approximate the unknown coefficient \( q \) using a fully connected neural network with Fourier feature embedding. The network takes as input \((t, x) = (t, x_1, x_2)\) for time-dependent coefficients or simply \( x \) for time-independent coefficients. In general, we represent the unknown coefficient in the following neural network form:
\begin{equation}\label{nn_cutoff}
    q_{\mathrm{NN}}(t,x) = \mathrm{MLP}(\mathcal{F}(t,x_1,x_2)) \cdot \chi_{\alpha,\beta}(t). 
\end{equation}

To improve neural network stability with high‑frequency information and better capture the unknown coefficient’s spatiotemporal structure, we first map the inputs \((t,x_1,x_2)\) into a high‑dimensional Fourier feature space by the Fourier transform \(\mathcal{F}\) in (\ref{nn_cutoff}). This mapping consists of two parts: basis frequencies and interaction frequencies.
For each input dimension \(z \in \{t, x_1, x_2\}\), we define:
\[
\mathcal{F}(z) = \left[\sin(k\omega_z z), \cos(k\omega_z z)\right]_{k=0}^{n_{\mathrm{basis}}},
\]
where \(\omega_z = 2\pi / L_z\) and \(L_z\) is the domain length in \(z\). Notice that single‑variable basis functions alone cannot capture spatiotemporal couplings in \(q\). Hence, for time‑dependent coefficients, we add cross terms to \(\mathcal{F}\):
\[
\sin(k'\omega_{z_1}z_1)\sin(k'\omega_{z_2}z_2),\quad
\cos(k'\omega_{z_1}z_1)\cos(k'\omega_{z_2}z_2),\quad
\sin(k'\omega_{z_1}z_1)\cos(k'\omega_{z_2}z_2),
\]
with \(k' = 1,\dots,n_{\mathrm{cross}}\), \(z_1 \neq z_2\), and \(z_1,z_2 \in \{t,x_1,x_2\}\).
The choice of \(n_{\mathrm{basis}}\) controls frequency resolution per dimension, while \(n_{\mathrm{cross}}\) controls the range of inter‑dimensional couplings. Without cross terms, the input features are separable. Nonlinear layers can approximate couplings, but they tend to produce additive‑like fits that are less effective for complex spatiotemporal patterns. Explicit cross terms give the network direct access to these couplings. Moreover, setting \(n_{\mathrm{cross}} < n_{\mathrm{basis}}\) acts as implicit regularization: it forces the network to prioritize dominant single‑variable terms and low‑frequency couplings, while excluding high‑frequency interaction features that are noise‑sensitive. This selective frequency restriction not only optimizes the feature space but also improves robustness under noisy conditions.
Then, the features in $\mathcal{F}$ are passed through a multi-layer perceptron (MLP) with 3 fully-connected hidden layers, each containing $d_{\mathrm{hidden}}$ neurons with $\tanh$ activation functions, followed by a linear output layer producing a single scalar $\mathrm{MLP}(\mathcal{F}(t,x_1,x_2))$ in (\ref{nn_cutoff}).

For the time-dependent case, the finite speed of wave propagation has a more significant effect on the recovery. Boundary observations contain no information about \( q \) outside the region (\ref{regionD}). Due to this geometric constraint, we apply a temporal truncation \( \chi_{\alpha,\beta}(t) \) to the network output (\ref{nn_cutoff}), where \( \chi_{\alpha,\beta}(t) \) is a \( C_c^\infty \) cutoff function with support in \( [\alpha T, \beta T] \). This ensures that the reconstructed coefficient vanishes near \( t = 0 \) and \( t = T \), and that the constraint is enforced smoothly without disrupting the Adam optimizer. For time-independent coefficients, the time-dependent component and the temporal truncation \( \chi_{\alpha,\beta}(t) \) can be removed.

\subsection{Experimental Settings}\label{se8.3}
For the sake of simplicity, the computational domain considered in this section is \((0,T) \times \Omega\), where \(T = 10\) and \(\Omega = (0,4) \times (0,4)\) is a bounded rectangular region. To solve the forward problem numerically, we use a finite difference method with second‑order central differences for both spatial and temporal derivatives.
For the time‑independent case, we employ a fine grid of size \(320 \times 80 \times 80\); for the time‑dependent case, we use a coarse grid of size \(160 \times 40 \times 40\). The finer grid is chosen because the time‑independent example exhibits more complex spatial structure (e.g., high‑frequency oscillations). In contrast, to demonstrate the performance of the time‑dependent reconstruction at lower computational cost, we select coefficients with relatively simple spatiotemporal structures.

The inverse problem is formulated as recovering \(q\) from boundary observations. Following the setting in \cite{harju2025}, we construct plane-wave Neumann boundary conditions of the form:
\begin{equation*}
    f(t,x) =\partial_\nu\left( A_0 \cdot \varphi_h(\zeta) \cdot \tau^{1/2} \exp(-\tau \zeta^2 / 2)\right),
\end{equation*}
where \(\zeta = (t-t_0) - \cos\theta\,(x_1-x_{10}) - \sin\theta\,(x_2-x_{20})\) is the phase variable, \(\varphi_h\) is a smooth cutoff function with support width \(h=1.5\), \(\tau=700\) is the frequency parameter, \(A_0=0.1\) is the amplitude, and \(\theta\) specifies the propagation direction. We use \(N_{\mathrm{obs}} = 15\) observations, with \(\theta^{(i)} = \frac{\pi}{2} \cdot \frac{i-1}{2}\) for \(i = 1, 2, 3\), spatial origin at \((x_{10}, x_{20}) = (-0.5, -0.5)\), and starting times \(t_0^{(j)} = 0.15j\,T\) for \(j = 0, 1, 2, 3, 4\).

In order to realize the first-order linearized NtD map (\ref{11-3}), we take a standard approximation approach used in the coefficient recovery for semilinear Helmholtz equation in the frequency domain as shown in \cite{LuSaloXu}. In particular, the linearized Dirichlet boundary data is approximated by
\begin{align*}
 \Lambda_q'f := \partial_\gamma(\Lambda_{\gamma q})|_{\gamma=0}f
    = u^{(1)}|_{(0,T)\times\partial\Omega} \approx u|_{(0,T)\times\partial\Omega} - u^{(0)}|_{(0,T)\times\partial\Omega},  
\end{align*}
where $u|_{(0,T)\times\partial\Omega}$ satisfies the original semilinear wave equation (\ref{2-1}) and $u^{(0)}|_{(0,T)\times\partial\Omega}$ satisfies the linear wave equation (\ref{11-4}). 
In the reconstruction algorithm, we do not implement the PIE identity but include all the linearized Dirichlet data into the loss function. More precisely, by denoting the parameter in the neural network by $\Theta$, we consider the following minimization approach such that 
\begin{equation*}
    \mathcal{L}(\Theta) = \frac{1}{N_{\mathrm{obs}}} \sum_{r=1}^{N_{\mathrm{obs}}} \frac{1}{|\partial\Omega|} \int_{(0,T)\times\partial\Omega} |u_{\Theta,r}^{(1)} - (u_{\mathrm{obs},r}-u_{\mathrm{obs},r}^{(0)})|^2 \, dS_x dt,
\end{equation*}
where $N_{\mathrm{obs}}=15$ is total number of the measurement near the boundary. The above $u_{\mathrm{obs},r}$ is the solutions of the 
semilinear wave equation 
\begin{equation*}
    \begin{cases}
        \square u(t,x) 
        +qu^m(t,x) = 0, & \mathrm{in}\ (0,T)\times\Omega, \\
        \partial_\nu u(t,x) = f(t,x), & \mathrm{on}\ (0,T)\times\partial\Omega, \\
        u(0,x)=\partial_t u(0,x)=0, & x\in\Omega,
    \end{cases}
\end{equation*}
and $u_{\mathrm{obs},r}^{(0)}$ is the solution of the linear wave equation
\begin{equation*}
    \begin{cases}
        \square u^{(0)}= 0, & \mathrm{in}\ (0,T)\times\Omega, \\
        \partial_\nu u^{(0)} = f, & \mathrm{on}\ (0,T)\times\partial\Omega, \\
        u^{(0)}(0,x)=\partial_t u^{(0)}(0,x)=0, & x\in\Omega,
    \end{cases}
\end{equation*}
respectively. The variable $u_{\Theta,r}^{(1)}$ is the solution computed by the linearized wave equation given the coefficient function $q_{NN}$, such that 
\begin{equation*}
    \begin{cases}
        \square u^{(1)} + q_{NN} (u_{\mathrm{obs},r}^{(0)})^m = 0, & \mathrm{in}\ (0,T)\times\Omega, \\
        \partial_\nu u^{(1)} = 0, & \mathrm{on}\ (0,T)\times\partial\Omega, \\
        u^{(1)}(0,x)=\partial_t u^{(1)}(0,x)=0, & x\in\Omega.
    \end{cases}
\end{equation*}
In the numerical examples, we also added $2\%$ of relative error with respect to the exact data which consists with those in  \cite{harju2025}.

\subsection{Reconstruction Results}
To evaluate the reconstruction algorithm, we design coefficient functions for both time-independent and time-dependent cases. In both examples, we employ the Adam optimizer with a learning rate of $0.01$ and a maximum of $5000$ iterations in neural network training. 

\begin{figure}[htp]
    \centering
    \includegraphics[width=0.9\textwidth]{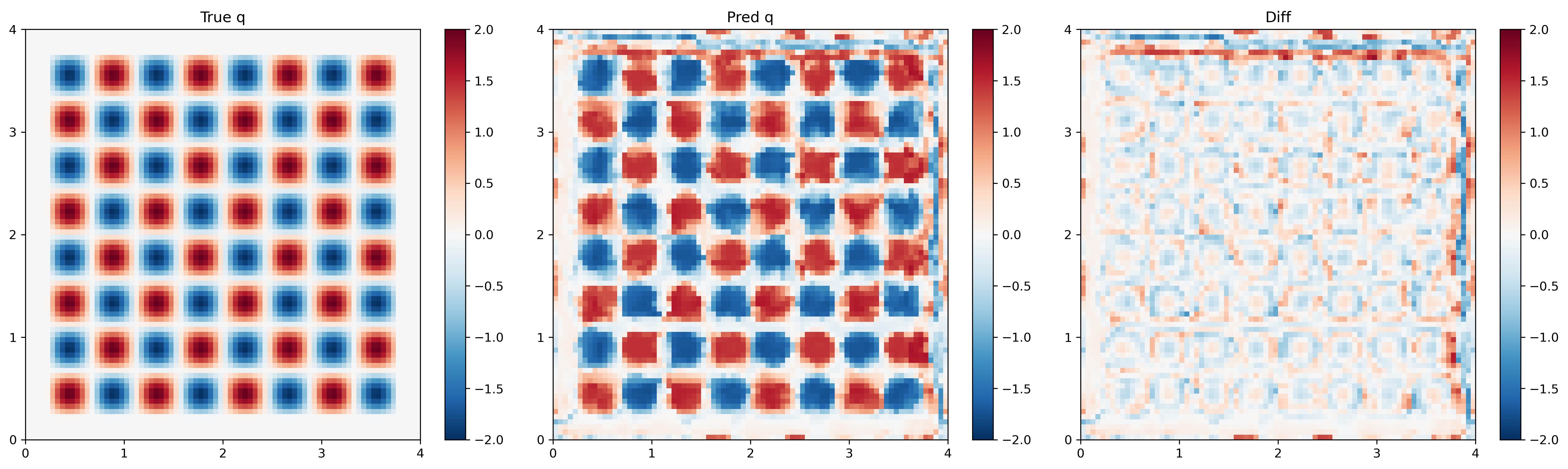}
    \caption{Time-independent case: exact coefficient $q_1$ (left panel); reconstructed neural network approximation (middle panel); absolute error (right panel). }
    \label{fig:q1_comparison}
\end{figure}

For the time-independent case, following \cite{harju2025}, we consider
\begin{equation}\label{tf1}
    q_1(x) = \sin(4\pi k \tilde{x}_1)\sin(4\pi k \tilde{x}_2)\mathbf{1}_{\mathrm{square}}(x_1,x_2), 
\end{equation}
which is an oscillatory function with $\tilde{x}_s = (x_s-2)/4$, $s=1,2$, and $\mathbf{1}_{\mathrm{square}}$ denotes the characteristic function of a square set in $\mathbb{R}^2$. We set $k=4$ to demonstrate the effectiveness of the algorithm and choose the neural network parameters as $(10, 40)$. Figure \ref{fig:q1_comparison} compares the true time-independent coefficient function $q_1$ defined in (\ref{tf1}) with its reconstructed neural network approximation. The algorithm accurately recovers the main structure of the oscillatory pattern, including the position, shape, and size of $q_1$, achieving performance similar to that reported in \cite{harju2025}. It should be noted, however, that due to the choice of a Fourier feature space, the approximation of the sharp jump in the coefficient $q_1$ is not highly accurate.

\begin{figure}
    \centering
    \includegraphics[width=0.9\textwidth]{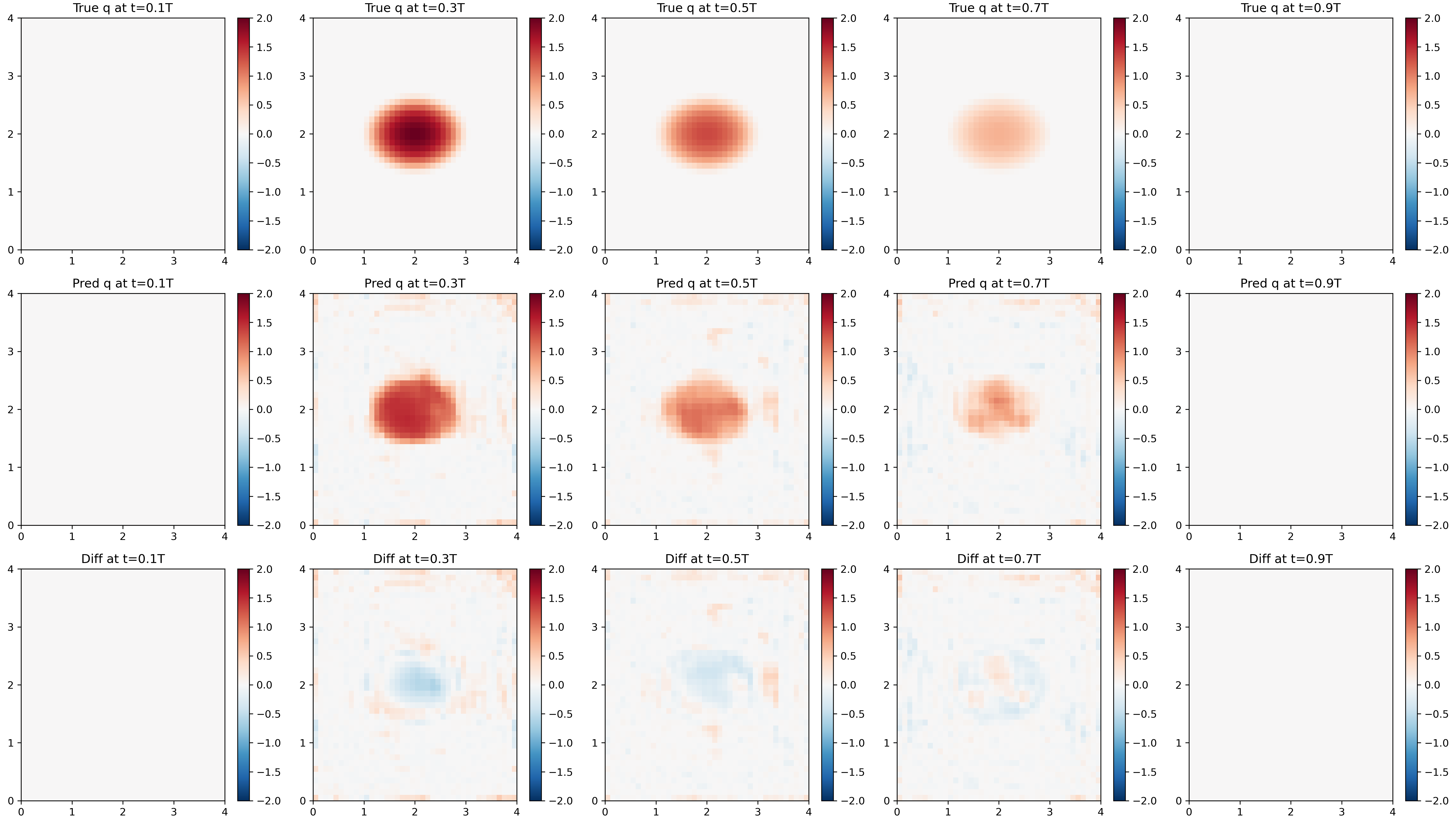}
    \caption{True and reconstructed coefficients for \(q_2\) in case (i). Top row: exact coefficient at different time snapshots. Middle row: neural network approximation. Bottom row: absolute error.
    }
    \label{fig:time_dep_case1}
\end{figure}

\begin{figure}
    \centering
    \includegraphics[width=0.9\textwidth]{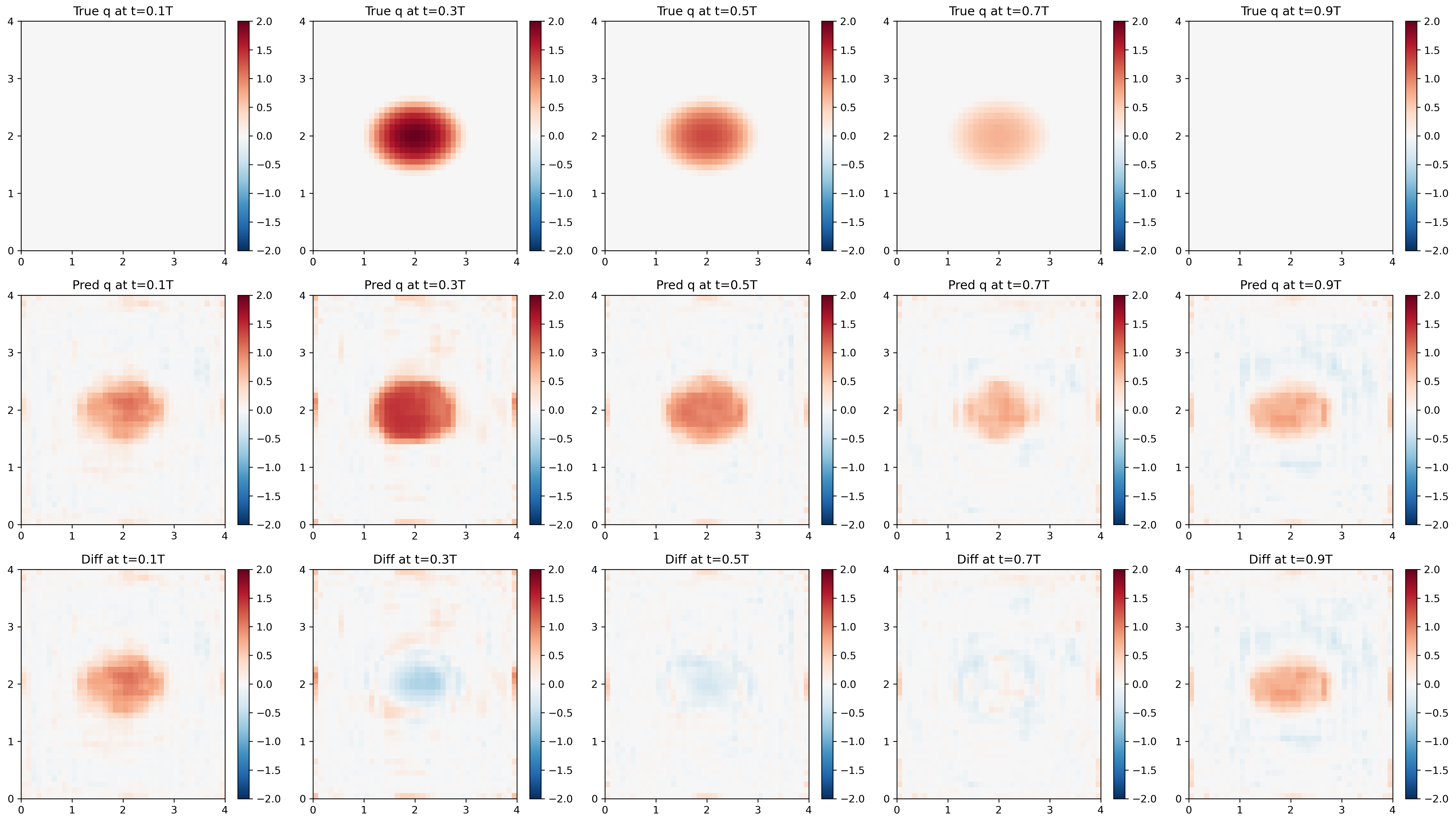}
    \caption{True and reconstructed coefficients for \(q_2\) in case (ii). Top row: exact coefficient at different time snapshots. Middle row: neural network approximation. Bottom row: absolute error.
    }
    \label{fig:time_dep_case2}
\end{figure}

\begin{figure}
    \centering
    \includegraphics[width=0.9\textwidth]{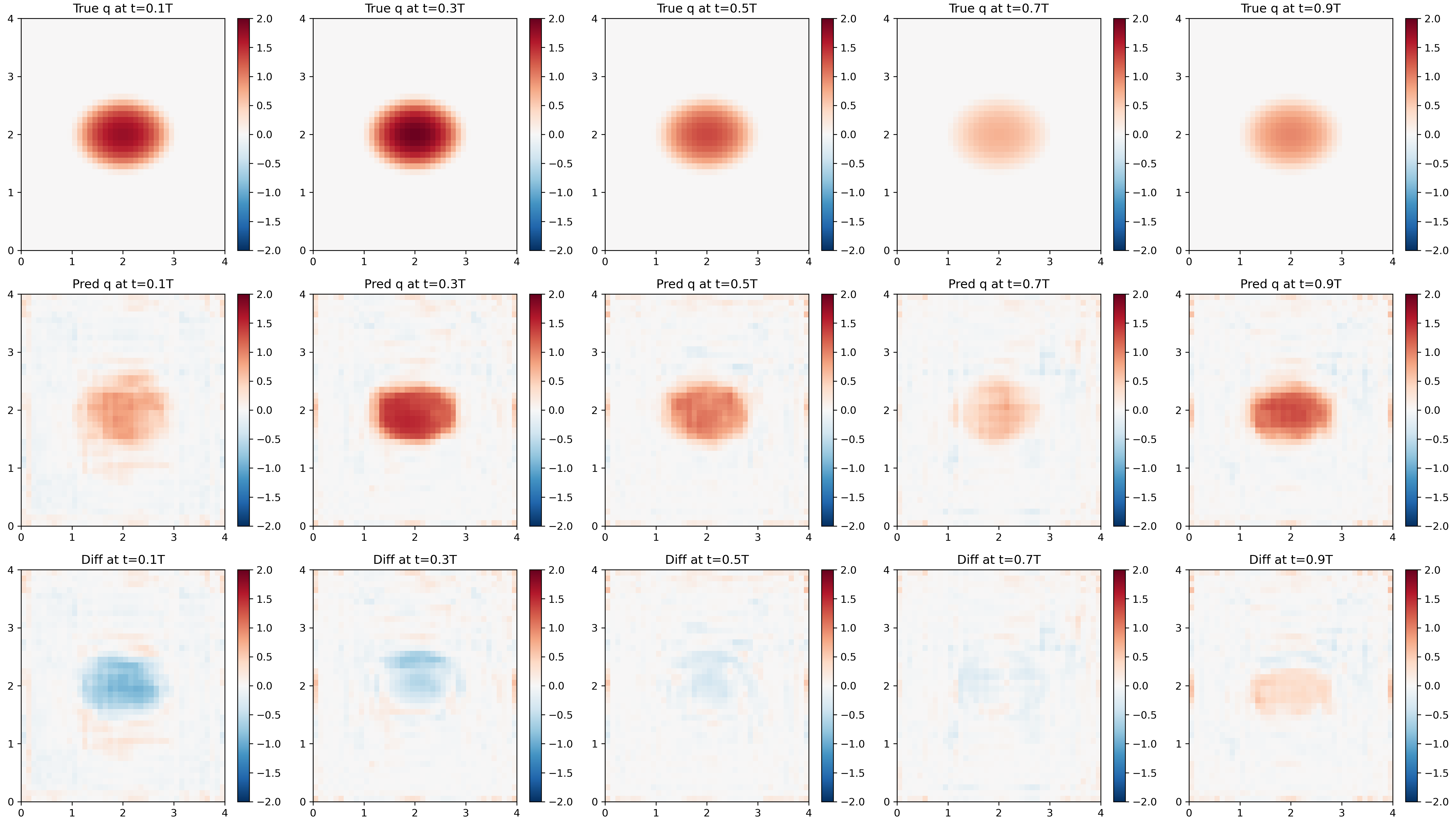}
    \caption{True and reconstructed coefficients for \(q_2\) in case (iii). Top row: exact coefficient at different time snapshots. Middle row: neural network approximation. Bottom row: absolute error.
    }
    \label{fig:time_dep_case3}
\end{figure}

For the time-dependent case, we consider a smooth coefficient function $q_2(t,x)$ of the form
\begin{equation*}
    q_2(t,x) = \varphi_{\mathrm{ellipse}}(x)\bigl(1 + A\sin(\omega t)\bigr)\chi_{\alpha',\beta'}(t),
\end{equation*}
where $\varphi_{\mathrm{ellipse}}$ is a smooth bump function as defined in \cite{harju2025}, given by
\begin{equation*}
    \varphi_{\mathrm{ellipse}}(x) =
    \begin{cases}
        \exp\left(\dfrac{1}{(x_1/a)^2 + (x_2/b)^2 - 1}\right), &  (x_1/a)^2 + (x_2/b)^2 < 1, \\
        0, & \text{otherwise},
    \end{cases}
\end{equation*}
with $A=0.5$, $\omega = 2\pi/T$, and $\chi_{\alpha',\beta'}(t)$ as defined in (\ref{nn_cutoff}). %We note that in \cite{harju2025}, the authors consider a time-independent bump function as the unknown coefficient.  
For the time-dependent smooth coefficient $q_2(t,x)$, we increase the neural network parameters to $(n_{\mathrm{basis}}, d_{\mathrm{hidden}}) = (15, 50)$ and set $n_\mathrm{cross}=3$, as the additional temporal dimension requires more representational capacity. To study the influence of temporal truncation, we consider the following three settings: 
\begin{enumerate}
    \item[(i)] Both the true coefficient $q_2$ and the network output are truncated in time;
    \item[(ii)] Only $q_2$ is truncated, while the neural network approximation is unconstrained;
    \item[(iii)] Neither $q_2$ nor the neural network approximation is truncated.
\end{enumerate}
The results are shown in Figures \ref{fig:time_dep_case1}--\ref{fig:time_dep_case3}. In case (i), with $(\alpha,\beta) = (0.2,0.8)$ and $(\alpha',\beta') = (0.25,0.75)$, the reconstructed coefficient captures both the spatial structure and the temporal variation. In case (ii), with $(\alpha,\beta) = (0,1)$ and $(\alpha',\beta') = (0.25,0.75)$, the neural network approximation produces artifacts outside the region reachable by finite-speed wave propagation. In case (iii), with $(\alpha,\beta) = (\alpha',\beta') = (0,1)$, the reconstruction quality degrades further. These comparisons confirm that temporal truncation of the neural network approximation is necessary for stable time-dependent reconstruction, a requirement induced by the finite speed of wave propagation.

\section*{Acknowledgement}
This work was supported by National Key Research and Development Programs of China (No. 2023YFA1009103), NSFC (No. 92570106) and Science and Technology Commission of Shanghai Municipality (No. 23JC1400501). Both authors would like to thank Jian Zhai (Fudan University) for his careful reading and valuable suggestions, which have significantly improved the paper.

%% 20250703 12:15 修改至此
\appendix  % 开始附录部分
% 重新定义附录章节的格式
\renewcommand{\thesection}{Appendix \Alph{section}}  % 使附录的编号显示为 Appendix A, Appendix B
\section{Auxiliary Results}

\textbf{Proof of Lemma \ref{Lem:4}.}
    For simplicity's sake, we prove the estimate when 
    $y_0=\vec{0}$, since it can be later applied to 
    $b(y+y_0)$ in place of $b(y)$. By the definition of $\chi_\delta(y)$ 
    in (\ref{cutoff}), the integrals
    \begin{equation*}
        \int_{(0,T)\times\Omega'}
        \prod_{j=0}^{m}a_{j0}(y)dy,\qquad
        \int_{(0,T)\times\Omega'}|y|
        \prod_{j=0}^{m}a_{j0}(y)dy
    \end{equation*}
    exist. To ensure that the constant coefficient has a uniform bound, it needs to 
    prove that the integrals have uniform positive 
    lower bounds; thus, to prove that the volume of 
    the support of the integrand has a uniform positive lower bound.
%     It is trivial when $m=2$, thus we only need to prove the case 
%     $m\ge3$.

    According to the construction of the geometric optics, the 
    support of the integrand is the intersection of 
    hypercylinders $C_j$, $j=0,1,2$, when $m=2$, or 
    $j=0,1,2,3$, when $m\ge3$, 
    with axis directions as $\xi_j^\sharp$, and 
    radius $\delta$, restricted to $(0,T)\times\Omega'$, and 
    \begin{equation}\label{thetaij}
        \theta_{i,j}:=\angle(\xi_i^\sharp,\xi_j^\sharp)\ge\sigma_{\min}>0,\ \forall i\ne j.
    \end{equation}
    First, we may assume, without loss of generality, that any
    $\xi_j^\sharp$ is unit in this proof, and that all axes of these hypercylinders 
    intersect at the origin (since translation does not affect 
    volume). We restrict our analysis to the unit ball $B(0,1)\subseteq\mathbb{R}^{1+n}$, 
    as any intersection outside this region can be scaled appropriately.
    By the construction of the geometric optics solutions in (\ref{m3xij}), it is known that for $m=2$, three hypercylinders intersect, while for $m\ge3$, four hypercylinders intersect. Thus,
    we prove that the volume of $K=\bigcap_{j=0}^kC_j\cap B(0,1)$, $k=2$ when 
    $m=2$ and $k=3$ when $m\ge3$,  
    % \textbf{[Q34: Why only consider such cases?]}
    satisfies
    \begin{equation*}
        c_1\delta^{1+n}\le\mathrm{Vol}(K)\le c_2\delta^{1+n},
    \end{equation*}
    where $c_1,c_2>0$ depend only on $\sigma_{\min}$ and $n$.
    %\fix{Change "Chapter 6" to concrete equation. }

    Noticing that there exists a $(1+n)$-dimensional ball-shaped neighborhood, with radius $\delta$, contained in $K$, 
    thus there exists $c_1>0$ such that $\mathrm{Vol}(K)\ge c_1\delta^{1+n}$.
    % \textbf{[Q35: ball of what size?]}. 
    Considering the case that two cylinders intersect, there exists $c>0$ such that $\mathrm{Vol}(K)=c\delta^{1+n}/\sin\theta_{0,1}$,
    % \textbf{[Q36: what does it mean by ``of order"?]}
    due to the angle $\theta_{0,1}$ 
    between the axes of two cylinders. By (\ref{thetaij}), $1/\sin\theta_{0,1}$ 
    has an upper bound. When more cylinders intersect, the 
    intersection will be smaller, so there exists $c_2>0$ such that 
    $\mathrm{Vol}(K)\le c_2\delta^{1+n}$.

    The volume scales as $\delta^{1+n}$, with constants depending only on $\theta_0$ and $n$:
    $$
    \text{Vol}(K) \sim C(\sigma_{\min}, n) \delta^{1+n}.
    $$
    This proves both the lower and upper bounds uniformly for small $\delta$.
    We then write
    \begin{equation*}
        C(n)\delta^{1+n}:=\int_{(0,T)\times\Omega'}
        \prod_{j=0}^{m}a_{j0}(y)dy,\quad
        C''(n)\delta^{1+n}:=\int_{(0,T)\times\Omega'}|y|
        \prod_{j=0}^{m}a_{j0}(y)dy,
    \end{equation*}
    and the estimate% whose order is independent on $\delta$:
    \begin{align*}
        C(n)&=\int_{\mathbb{R}^{1+n}}
        \prod_{j=0}^{m}\left(\chi_1\left(\xi_j\cdot y\right)
        \prod_{l=1}^{n-1}\chi_1\left(w_{jl}\cdot y\right)\right)dy,\\
        C''(n)&=\int_{\mathbb{R}^{1+n}}|y|
        \prod_{j=0}^{m}\left(\chi_1\left(\xi_j\cdot y\right)
        \prod_{l=1}^{n-1}\chi_1\left(w_{jl}\cdot y\right)\right)dy.
    \end{align*}

    Noting that 
    \begin{equation*}
        \left|b(0)-b(\delta y)\right|
        \le\|b\|_{C^1}\delta|y|,\quad \forall y\in
        \mathbb{R}^{1+n}, 
    \end{equation*}
    we immediately deduce
    \begin{align*}
        &\quad \left| b(0) -C(n,0)\delta^{-(n+1)}\int_{(0,T)\times\Omega'}
        b(y)\prod_{j=0}^{m}\left(\chi_\delta\left(\xi_j\cdot y\right)
        \prod_{l=1}^{n-1}\chi_\delta\left(w_{jl}\cdot y\right)\right)dy\right|\\
        % &= \left| b(0) -C(n,0)\int_{(0,T/\delta)\times(\Omega'/\delta)}
        %     b(\delta y)
        %     \prod_{j=0}^{m}\left(\chi_1\left(\xi_j\cdot y\right)
        %     \prod_{l=1}^{n-1}\chi_1\left(w_{jl}\cdot y\right)\right)dy\right| \\
        &= \left| C(n,0)\int_{(0,T/\delta)\times(\Omega'/\delta)}
            \left(b(0) -b(\delta y)\right)
            \prod_{j=0}^{m}\left(\chi_1\left(\xi_j\cdot y\right)
            \prod_{l=1}^{n-1}\chi_1\left(w_{jl}\cdot y\right)\right)dy\right| \\
        % &\le C(n,0)\delta\|b\|_{C^1}
        %     \int_{(0,T/\delta)\times(\Omega'/\delta)}
        %     |y|\prod_{j=0}^{m}\left(\chi_1\left(\xi_j\cdot y\right)
        %     \prod_{l=1}^{n-1}\chi_1\left(w_{jl}\cdot y\right)\right)dy\\
        &\le C'(n)\|b\|_{C^1}\delta.
    \end{align*}
    \qed

% \end{proof-of-lemma4}
% \begin{Lem}
%     Let $G\in C^1((0,T)\times\Omega)$, $y_0=(t_0,x_0)\in
%     (0,T)\times\partial\Omega$ and $\delta>0$. Then there exist constants 
%     $C(n,y_0)>0$ and $C'(n)>0$ such that the following estimate holds:
%     \begin{equation*}
%         \left| \mathcal{L}(G)(y_0,\xi^\sharp) -C(n,y_0)\delta^{-n}\int_{(0,T)\times\Omega}
%         G(y)\prod_{j=0}^{2}\left(\chi_\delta\left(\xi_j\cdot(y-y_0)\right)
%         \prod_{l=1}^{n-1}\chi_\delta\left(w_{jl}\cdot(y-y_0)\right)\right)dy\right|\le
%         C'(n)\|G\|_{C^1}\delta,
%     \end{equation*}
%     where 
%     all $\xi_j$ and $w_{jl}$ are defined as in equations
%     (\ref{12-7}) and (\ref{12-9}). In particular, the integral 
%     on the left-hand side converges uniformly to $\mathcal{L}(G)$ as
%     $\delta\to0^+$.
% \end{Lem}
% \begin{proof-of-lemma5}
\textbf{Proof of Lemma \ref{Lem:5}.}
    Let $\Omega_T':=(-\delta,T+\delta)\times(\Omega+\delta)$, where 
    $\Omega+\delta:=\bigcup_{x\in\Omega}B_{\delta}(x)$. Then for any point 
    $y=(t,x)\in\gamma_{y_0,\xi}\cap(0,T)\times\overline{\Omega}$, we have 
    $B_\delta(y)\subset\Omega_T'$. Define the tubular neighborhood
    % \begin{equation*}
    %     H_\delta:=\bigcup_{y\in\gamma_{y_0,\xi}\cap(0,T)\times\overline{\Omega}}B_\delta(y).
    % \end{equation*}
    \begin{equation*}
        H_\delta:=\left\{ y_0+s\xi^\sharp+z\ |\ y_0+s\xi^\sharp\in(0,T)\times\overline{\Omega},z\in(\xi^\sharp)^\perp,|z|<\delta\right\},
    \end{equation*}
    where $(\xi^\sharp)^\perp$ denotes the orthogonal complement of $\xi^\sharp$ in Euclidean 
    sense.

    Since $\Omega$ is convex and simply connected, 
    the ray $\gamma_{y_0,\xi}\cap\overline{(0,T)\times\Omega}$ is a line 
    segment between two boundary points, denoted as
    $(t_0,x_0)$ and $(t_1,x_1)$. Let us assume this interval is 
    parametrized by \( s \in [0, t_1 - t_0] \), where \( t_1 > t_0 \).
    
    By Lions' extension theorem, there exists an
    extension $F\in C^1(\mathbb{R}^{1+n})$ such that
    \begin{equation*}
        F|_{(0,T)\times\Omega}=G,\quad\text{and}\quad \|F\|_{C^1(\mathbb{R}^{1+n})}\le C\|G\|_{C^1((0,T)\times\Omega)}.
    \end{equation*}

    We write \( y = y_0 + s\xi^\sharp + z \), with 
    \( s = (y - y_0)\cdot \xi^\sharp \), and \( z \in (\xi^\sharp)^\perp \). 
%     Let us denote the integrand as
% \[
% G_{y_0,\xi} :=  
% G(y) \prod_{j=0}^{2}\left( \chi_\delta(\xi_j \cdot (y - y_0))
% \prod_{l=1}^{n-1} \chi_\delta(w_{jl} \cdot (y - y_0))\right).
% \]
    % !!!and $[(0,T)\times\Omega]\cap\mathrm{supp}F=\mathrm{supp}(G_{y_0,\xi})$
    Changing variables \( y \mapsto (s, z) \) gives \( dy = ds \, dz \), 
    and we derive
    \begin{align*}
        \int_{(0,T)\times\Omega}
          &   G(y)\prod_{j=0}^{2}a_{j0}(y-y_0)dy \\
        &\quad =\left(\int_{((0,T)\times\Omega)\cap H_\delta}+\int_{[(0,T)\times\Omega]\cap\mathrm{supp}F\setminus H_\delta}\right)
            G(y)\prod_{j=0}^{2}a_{j0}(y-y_0)dy\\
        &\quad =\left(\int_{H_\delta}-\int_{H_\delta\setminus (0,T)\times\Omega}+\int_{[(0,T)\times\Omega]\cap\mathrm{supp}F\setminus H_\delta}\right)
            F(y)\prod_{j=0}^{2}a_{j0}(y-y_0)dy. 
    \end{align*}
    Now, we write the main term as
    \begin{align*}
        \int_{H_\delta}
            F(y)\prod_{j=0}^{2}a_{j0}(y-y_0)dy
        &=\int_0^{t_1-t_0}\int_{z\in(\xi^\sharp)^\perp}
            F(y_0+s\xi^\sharp+z)\prod_{j=0}^{2}a_{j0}(s\xi^\sharp+z)dzds\\
        % &=\int_0^{t_1-t_0}\int_{z\in(\xi^\sharp)^\perp}
        %     F(y_0+s\xi^\sharp+z)\prod_{j=0}^{2}a_{j0}(z)dzds\\
        &=\int_{z\in(\xi^\sharp)^\perp}\left(\int_0^{t_1-t_0}
            F(y_0+s\xi^\sharp+z)ds\right)\prod_{j=0}^{2}a_{j0}(z)dz\\
        &:=\int_{z\in(\xi^\sharp)^\perp}L(F)(z)\prod_{j=0}^{2}a_{j0}(z)dz. 
    \end{align*}
    Here, we notice that
    $L(F)(0)=\mathcal{L}(F)(y_0,\xi^\sharp)=\mathcal{L}(G)(y_0,\xi^\sharp)$.
    The error terms can be estimated as
    \begin{equation*}
        \left|\int_{H_\delta\setminus (0,T)\times\Omega}
            F(y)\prod_{j=0}^{2}a_{j0}(y-y_0)dy\right|
        \le C\|F\|_{C^1(\mathbb{R}^{1+n})}\cdot\mathrm{vol}(H_\delta\setminus (0,T)\times\Omega)
        \le C\delta^{1+n},
    \end{equation*}
    and
    \begin{align*}
       &  \left|\int_{[(0,T)\times\Omega]\cap\mathrm{supp}F\setminus H_\delta}
            F(y)\prod_{j=0}^{2}a_{j0}(y-y_0)dy\right|  \\
       & \quad  \le C\|F\|_{C^1(\mathbb{R}^{1+n})}\cdot\mathrm{vol}([(0,T)\times\Omega]\cap\mathrm{supp}F\setminus H_\delta)
        \le C\delta^{1+n}.
    \end{align*}
    We know that $(\xi^\sharp)^\perp$ is isomorphic to $\mathbb{R}^n$, so 
    the proof will be similar to that of Lemma \ref{Lem:4} if we perform proper 
    coordinate transform.
    % and show $\|\mathcal{L}(F)(\cdot,\xi)\|_{C^1}$ %that
    %\begin{equation*}
    %    \sup_{y\in\mathbb{R}^{1+n}}\left|\mathcal{L}(G)(y,\xi)\right|
    %    +\sup_{y\in\mathbb{R}^{1+n}}\left|\partial_y\mathcal{L}(G)(y,\xi)\right|
    %    :=L(\xi)
    %\end{equation*}
    % is uniformly bounded in $\xi\in S_1^n$. But this follows, since 
    % $F$ is compactly supported and smooth so that by Fubini's 
    % theorem we can change the order of differentiation and integration. 
    By Lemma \ref{Lem:4}, we know that
    \begin{equation*}
        \left|L(F)(0) -C(n,y_0)\delta^{-n}\cdot
          \int_{H_\delta}
        F(y)\prod_{j=0}^{2}a_{j0}(y-y_0)dy\right|\le C\|F\|_{C^1}\delta.
    \end{equation*}
    Therefore, 
    \begin{align*}
        &\left| \mathcal{L}(G)(y_0,\xi^\sharp) -C(n,y_0)\delta^{-n}\int_{(0,T)\times\Omega}
        G(y)\prod_{j=0}^{2}a_{j0}(y-y_0)dy\right|\\
        &\le\left| \mathcal{L}(G)(y_0,\xi^\sharp)-\int_0^{t_1-t_0}
            F(y_0+s\xi^\sharp)ds\right|\\
        &\quad + \left|L(F)(0) -C(n,y_0)\delta^{-n}\cdot
         \left( \int_{H_\delta}-\int_{H_\delta\setminus (0,T)\times\Omega}+\int_{[(0,T)\times\Omega]\cap\mathrm{supp}F\setminus H_\delta}\right)
        F(y)\prod_{j=0}^{2}a_{j0}(y-y_0)dy\right|.
    \end{align*}
    By Lemma \ref{Lem:4}, we know that the constant $C(n,y_0)$ has uniform 
    positive lower and upper bounds independent on $y_0$,
    and therefore we have completed the proof.

 \qed   
% \end{proof-of-lemma5}

\bibliographystyle{plain}
\bibliography{refs}

\begin{thebibliography}{10}

\bibitem{Aless1}
Giovanni Alessandrini.
\newblock Stable determination of conductivity by boundary measurements.
\newblock {\em Appl. Anal.}, 27(1-3):153--172, 1988.

\bibitem{Alexakis2022}
Spyros Alexakis, Ali Feizmohammadi, and Lauri Oksanen.
\newblock Lorentzian {C}alder\'on problem under curvature bounds.
\newblock {\em Invent. Math.}, 229(1):87--138, 2022.

\bibitem{Alexakis2025}
Spyros Alexakis, Ali Feizmohammadi, and Lauri Oksanen.
\newblock Lorentzian {C}alder\'on problem near the {M}inkowski geometry.
\newblock {\em J. Eur. Math. Soc. (JEMS)}, 27(9):3771--3792, 2025.

\bibitem{Anderson2004}
Michael Anderson, Atsushi Katsuda, Yaroslav Kurylev, Matti Lassas, and Michael
  Taylor.
\newblock Boundary regularity for the {R}icci equation, geometric convergence,
  and {G}el\'fand's inverse boundary problem.
\newblock {\em Invent. Math.}, 158(2):261--321, 2004.

\bibitem{Matti3}
Yernat~M. Assylbekov and Ting Zhou.
\newblock Direct and inverse problems for the nonlinear time-harmonic {M}axwell
  equations in {K}err-type media.
\newblock {\em J. Spectr. Theory}, 11(1):1--38, 2021.

\bibitem{BaoZhang2014}
Gang Bao and Hai Zhang.
\newblock Sensitivity analysis of an inverse problem for the wave equation with
  caustics.
\newblock {\em J. Amer. Math. Soc.}, 27(4):953--981, 2014.

\bibitem{Matti5}
M.~I. Belishev.
\newblock An approach to multidimensional inverse problems for the wave
  equation.
\newblock {\em Dokl. Akad. Nauk SSSR}, 297(3):524--527, 1987.

\bibitem{Matti6}
Michael~I. Belishev and Yaroslav~V. Kurylev.
\newblock To the reconstruction of a {R}iemannian manifold via its spectral
  data ({BC}-method).
\newblock {\em Comm. Partial Differential Equations}, 17(5-6):767--804, 1992.

\bibitem{Matti9}
A.~L. Bukhgeĭm and M.~V. Klibanov.
\newblock Uniqueness in the large of a class of multidimensional inverse
  problems.
\newblock {\em Dokl. Akad. Nauk SSSR}, 260(2):269--272, 1981.

\bibitem{Busch2026}
Leonard Busch, Matti Lassas, Lauri Oksanen, and Mikko Salo.
\newblock On exponential instability of an inverse problem for the wave
  equation.
\newblock {\em Int. Math. Res. Not. IMRN}, (4):Paper No. rnag016, 11, 2026.

\bibitem{Matti11}
C\u at\u alin~I. C\^arstea, Gen Nakamura, and Manmohan Vashisth.
\newblock Reconstruction for the coefficients of a quasilinear elliptic partial
  differential equation.
\newblock {\em Appl. Math. Lett.}, 98:121--127, 2019.

\bibitem{Matti12}
Xi~Chen, Matti Lassas, Lauri Oksanen, and Gabriel Paternain.
\newblock Detection of {H}ermitian connections in wave equations with cubic
  non-linearity.
\newblock {\em J. Eur. Math. Soc. (JEMS)}, 24(7):2191--2232, 2022.

\bibitem{Lassas14}
Yvonne Choquet-Bruhat.
\newblock {\em General relativity and the {E}instein equations}.
\newblock Oxford Mathematical Monographs. Oxford University Press, Oxford,
  2009.

\bibitem{Zhai8}
Constantine~M. Dafermos and William~J. Hrusa.
\newblock Energy methods for quasilinear hyperbolic initial-boundary value
  problems. {A}pplications to elastodynamics.
\newblock {\em Arch. Rational Mech. Anal.}, 87(3):267--292, 1985.

\bibitem{deHoopetal.2018}
Maarten~V. de~Hoop, Paul Kepley, and Lauri Oksanen.
\newblock Recovery of a smooth metric via wave field and coordinate
  transformation reconstruction.
\newblock {\em SIAM J. Appl. Math.}, 78(4):1931--1953, 2018.

\bibitem{Eskin2007}
G.~Eskin.
\newblock Inverse hyperbolic problems with time-dependent coefficients.
\newblock {\em Comm. Partial Differential Equations}, 32(10-12):1737--1758,
  2007.

\bibitem{Feizmohammadietal.2021}
Ali Feizmohammadi, Joonas Ilmavirta, Yavar Kian, and Lauri Oksanen.
\newblock Recovery of time-dependent coefficients from boundary data for
  hyperbolic equations.
\newblock {\em J. Spectr. Theory}, 11(3):1107--1143, 2021.

\bibitem{AliLauri}
Ali Feizmohammadi and Lauri Oksanen.
\newblock Recovery of zeroth order coefficients in non-linear wave equations.
\newblock {\em J. Inst. Math. Jussieu}, 21(2):367--393, 2022.

\bibitem{harju2025}
Markus Harju, Suvi Anttila, and Teemu Tyni.
\newblock X-ray imaging from nonlinear waves: numerical reconstruction of a
  cubic nonlinearity, 2025.

\bibitem{Isakov1993}
V.~Isakov.
\newblock On uniqueness in inverse problems for semilinear parabolic equations.
\newblock {\em Arch. Rational Mech. Anal.}, 124(1):1--12, 1993.

\bibitem{IsakovSylvester1994}
Victor Isakov and John Sylvester.
\newblock Global uniqueness for a semilinear elliptic inverse problem.
\newblock {\em Comm. Pure Appl. Math.}, 47(10):1403--1410, 1994.

\bibitem{Kianetal.2019}
Yavar Kian, Yaroslav Kurylev, Matti Lassas, and Lauri Oksanen.
\newblock Unique recovery of lower order coefficients for hyperbolic equations
  from data on disjoint sets.
\newblock {\em J. Differential Equations}, 267(4):2210--2238, 2019.

\bibitem{KrupchykandUhlmann2020a}
Katya Krupchyk and Gunther Uhlmann.
\newblock Partial data inverse problems for semilinear elliptic equations with
  gradient nonlinearities.
\newblock {\em Math. Res. Lett.}, 27(6):1801--1824, 2020.

\bibitem{AliLauri14}
Yaroslav Kurylev, Matti Lassas, Lauri Oksanen, and Gunther Uhlmann.
\newblock Inverse problem for {E}instein-scalar field equations.
\newblock {\em Duke Math. J.}, 171(16):3215--3282, 2022.

\bibitem{Matti37}
Yaroslav Kurylev, Matti Lassas, and Gunther Uhlmann.
\newblock Inverse problems for {L}orentzian manifolds and non-linear hyperbolic
  equations.
\newblock {\em Invent. Math.}, 212(3):781--857, 2018.

\bibitem{Kurylevetal.2018b}
Yaroslav Kurylev, Lauri Oksanen, and Gabriel~P. Paternain.
\newblock Inverse problems for the connection {L}aplacian.
\newblock {\em J. Differential Geom.}, 110(3):457--494, 2018.

\bibitem{LaiLin2019}
Ru-Yu Lai and Yi-Hsuan Lin.
\newblock Global uniqueness for the fractional semilinear {S}chr\"odinger
  equation.
\newblock {\em Proc. Amer. Math. Soc.}, 147(3):1189--1199, 2019.

\bibitem{Lassas2018}
Matti Lassas.
\newblock Inverse problems for linear and non-linear hyperbolic equations.
\newblock In {\em Proceedings of the {I}nternational {C}ongress of
  {M}athematicians---{R}io de {J}aneiro 2018. {V}ol. {IV}. {I}nvited lectures},
  pages 3751--3771. World Sci. Publ., Hackensack, NJ, 2018.

\bibitem{LassasP8}
Matti Lassas, Tony Liimatainen, Leyter Potenciano-Machado, and Teemu Tyni.
\newblock Uniqueness, reconstruction and stability for an inverse problem of a
  semi-linear wave equation.
\newblock {\em J. Differential Equations}, 337:395--435, 2022.

\bibitem{MattiLorentzian}
Matti Lassas, Tony Liimatainen, Leyter Potenciano-Machado, and Teemu Tyni.
\newblock Stability and {L}orentzian geometry for an inverse problem of a
  semilinear wave equation.
\newblock {\em Anal. PDE}, 18(5):1065--1118, 2025.

\bibitem{lightray}
Matti Lassas, Lauri Oksanen, Plamen Stefanov, and Gunther Uhlmann.
\newblock The light ray transform on {L}orentzian manifolds.
\newblock {\em Comm. Math. Phys.}, 377(2):1349--1379, 2020.

\bibitem{AliLauri18}
Matti Lassas, Gunther Uhlmann, and Yiran Wang.
\newblock Determination of vacuum space-times from the einstein-maxwell
  equations, 2017.

\bibitem{LuSaloXu}
Shuai Lu, Mikko Salo, and Boxi Xu.
\newblock Increasing stability in the linearized inverse {S}chr\"odinger
  potential problem with power type nonlinearities.
\newblock {\em Inverse Problems}, 38(6):Paper No. 065009, 25, 2022.

\bibitem{MaSalo}
Shiqi Ma, Suman~Kumar Sahoo, and Mikko Salo.
\newblock The anisotropic {C}alder\'on problem at large fixed frequency on
  manifolds with invertible ray transform.
\newblock {\em J. Lond. Math. Soc. (2)}, 110(4):Paper No. e13006, 35, 2024.

\bibitem{Oksanenetal.2024}
Lauri Oksanen, Mikko Salo, Plamen Stefanov, and Gunther Uhlmann.
\newblock Inverse problems for real principal type operators.
\newblock {\em Amer. J. Math.}, 146(1):161--240, 2024.

\bibitem{Paternain}
Gabriel~P. Paternain, Mikko Salo, and Gunther Uhlmann.
\newblock {\em Geometric inverse problems---with emphasis on two dimensions},
  volume 204 of {\em Cambridge Studies in Advanced Mathematics}.
\newblock Cambridge University Press, Cambridge, 2023.
\newblock With a foreword by Andr\'as Vasy.

\bibitem{SaloZhong2012}
Mikko Salo and Xiao Zhong.
\newblock An inverse problem for the {$p$}-{L}aplacian: boundary determination.
\newblock {\em SIAM J. Math. Anal.}, 44(4):2474--2495, 2012.

\bibitem{Stefanov1989}
P.~D. Stefanov.
\newblock Inverse scattering problem for the wave equation with time-dependent
  potential.
\newblock {\em J. Math. Anal. Appl.}, 140(2):351--362, 1989.

\bibitem{StefanovYang}
Plamen Stefanov and Yang Yang.
\newblock The inverse problem for the {D}irichlet-to-{N}eumann map on
  {L}orentzian manifolds.
\newblock {\em Anal. PDE}, 11(6):1381--1414, 2018.

\bibitem{Sun2010}
Ziqi Sun.
\newblock An inverse boundary-value problem for semilinear elliptic equations.
\newblock {\em Electron. J. Differential Equations}, pages No. 37, 5, 2010.

\bibitem{SunandUhlmann1997}
Ziqi Sun and Gunther Uhlmann.
\newblock Inverse problems in quasilinear anisotropic media.
\newblock {\em Amer. J. Math.}, 119(4):771--797, 1997.

\bibitem{Taylor11}
Michael~E. Taylor.
\newblock {\em Partial differential equations {I}. {B}asic theory}, volume 115
  of {\em Applied Mathematical Sciences}.
\newblock Springer, New York, second edition, 2011.

\bibitem{AliLauri24}
Gunther Uhlmann and Yiran Wang.
\newblock Determination of space-time structures from gravitational
  perturbations.
\newblock {\em Comm. Pure Appl. Math.}, 73(6):1315--1367, 2020.

\bibitem{WangandZhou2019}
Yiran Wang and Ting Zhou.
\newblock Inverse problems for quadratic derivative nonlinear wave equations.
\newblock {\em Comm. Partial Differential Equations}, 44(11):1140--1158, 2019.

\bibitem{Zou}
Sen Zou, Shuai Lu, and Boxi Xu.
\newblock Increasing stability of the first order linearized inverse
  {S}chr\"odinger potential problem with integer power type nonlinearities.
\newblock {\em SIAM J. Appl. Math.}, 84(4):1868--1889, 2024.

\end{thebibliography}

\end{document}